\documentclass{daj}
\usepackage{amsmath,amsthm}
\usepackage{mathrsfs,amssymb,graphicx,verbatim,hyperref}
\usepackage{paralist,enumitem}
\usepackage{esint}

\usepackage{color}

\renewcommand{\setminus}{\smallsetminus}
\usepackage{ifthen}

\renewcommand{\t}{\vartheta}
\newcommand{\ifsodaelse}[2]{\ifthenelse{\isundefined{\SODAF}}{#2}{#1}}
\ifsodaelse    {\usepackage{ltexpprt}\usepackage{amsmath}}{}
\usepackage{mathrsfs,amssymb,graphicx,verbatim}
\usepackage{paralist}

\newcommand\remove[1]{}

\newcommand{\rnote}[1]{}
\newcommand{\jnote}[1]{}

\newcommand{\1}{\mathbf{1}}
\newcommand{\e}{\varepsilon}
\newcommand{\R}{\mathbb{R}}
\renewcommand{\P}{\mathbb{P}}
\newcommand{\E}{\mathbb{E}}

\newcommand{\N}{\mathbb{N}}

\renewcommand{\P}{\mathbb{P}}

\newcommand{\Lip}{\mathrm{Lip}}
\newcommand{\vol}{\mathrm{vol}}

\newcommand{\C}{\mathbb{C}}
\newcommand{\f}{\varphi}

\renewcommand{\d}{\delta}

\newcommand{\n}{\{1,\ldots,n\}}

\renewcommand{\subset}{\subseteq}
\newtheorem{theorem}{Theorem}
\newtheorem{lemma}[theorem]{Lemma}
\newtheorem{proposition}[theorem]{Proposition}

\newtheorem{corollary}[theorem]{Corollary}

\newtheorem{definition}[theorem]{Definition}

\newtheorem{remark}[theorem]{Remark}

\newtheorem{question}[theorem]{Question}
\newcommand{\eqdef}{\stackrel{\mathrm{def}}{=}}
\date{}
\renewcommand{\le}{\leqslant}
\renewcommand{\ge}{\geqslant}
\renewcommand{\leq}{\leqslant}

\newcommand\Z{{{\mathbb Z}}}

\include{psfig}

\renewcommand{\epsilon}{\varepsilon}

\theoremstyle{remark}

\newcommand{\F}{\mathcal{F}}

\newcommand{\bddlin}[0]{\mathscr{L}}

\newcommand{\supp}[0]{\operatorname{supp}}

\newcommand{\radem}[0]{\varepsilon}

\newcommand{\Norm}[2]{\|#1\|_{#2}}

\newcommand{\dd}{\mathrm{d}}

\dajAUTHORdetails{%
  title = {Quantitative Affine Approximation for UMD Targets}, 
  author = {Tuomas Hyt\"{o}nen, Sean Li, Assaf Naor},
  plaintextauthor = {Tuomas Hytonen, Sean Li, Assaf Naor},
  plaintexttitle = {Quantitative Affine Approximation for UMD Targets}, 
  runningtitle = {Quantitative Affine Approximation for UMD Targets},
  runningauthor = {Tuomas Hyt\"{o}nen, Sean Li, and Assaf Naor},
  copyrightauthor = {Tuomas Hyt\"{o}nen, Sean Li, and Assaf Naor},
  keywords = {Quantitative differentiation, unconditional martingale differences, Littlewood--Paley theory},
}   

\dajEDITORdetails{%
   year={2016},
   number={6},
   received={5 November 2015},   
   published={28 February 2016},  
   doi={10.19086/da.614},       
}   

\begin{document}

\begin{frontmatter}[classification=text]

\title{Quantitative Affine Approximation\\ for UMD Targets} 

\author[tuomas]{Tuomas Hyt\"{o}nen\thanks{Supported by ERC grant ANPROB and by the Academy of Finland through projects 130166, 133264 and CoE in Analysis and Dynamics.}}
\author[sean]{Sean Li\thanks{Supported by NSF grants CCF-0832795 and DMS-1303910.}}
\author[naor]{Assaf Naor\thanks{Supported by NSF grant CCF-0832795, BSF grant
2010021, the Packard Foundation and the Simons Foundation.}}

\begin{abstract}
It is shown here that if $(Y,\|\cdot\|_Y)$ is a Banach space in which martingale differences are unconditional (a UMD Banach space) then there exists $c=c(Y)\in (0,\infty)$ with the following property. For every $n\in \N$ and $\e\in (0,1/2]$, if $(X,\|\cdot\|_X)$ is an $n$-dimensional normed space with unit ball $B_X$ and $f:B_X\to Y$ is a $1$-Lipschitz function  then there exists an affine mapping $\Lambda:X\to Y$ and a sub-ball $B^*=y+\rho B_X\subseteq B_X$ of radius $\rho\ge \exp(-(1/\e)^{cn})$ such that $\|f(x)-\Lambda(x)\|_Y\le \e \rho$ for all $x\in B^*$. This estimate on the  macroscopic scale of affine approximability of vector-valued Lipschitz functions is an asymptotic improvement (as $n\to \infty$) over the best previously known bound even when $X$ is $\R^n$ equipped with the Euclidean norm and $Y$ is a Hilbert space.
\end{abstract}
\end{frontmatter}

\section{Introduction}

In what follows,  the unit ball of a normed space $(X,\|\cdot\|_X)$
is denoted $B_X\eqdef\{x\in X:\ \|x\|_X< 1\}$. For $p\in [1,\infty]$ and
$n\in \N$, the space $\R^n$ equipped with the $\ell_p$ norm is
denoted as usual by $\ell_p^n$. Given two metric spaces $(U,d_U)$
and $(V,d_V)$, the Lipschitz constant of a mapping $f:U\to V$ is
denoted $\|f\|_{\Lip}$. Throughout this article, given $a,b\in (0,\infty)$, the notations
$a\lesssim b$ and $b\gtrsim a$ mean that $a\le cb$ for some
universal constant $c\in (0,\infty)$. The notation $a\asymp b$
stands for $(a\lesssim b) \wedge  (b\lesssim a)$.


For $n\in \N$ and $\e\in (0,1)$ let $r_n(\e)\ge 0$ be the supremum
over those $r\in [0,1)$ such that for every  Lipschitz function
$f:B_{\ell_2^n}\to \ell_2$ there exists a linear mapping
$T:\ell_2^n\to \ell_2$, a vector $a\in \ell_2$, and a sub-ball
$B^*=x_0+\rho B_{\ell_2^n}\subseteq B_{\ell_2^n}$ of radius $\rho\ge
r$, such that
\begin{equation}\label{UAAP def intro B^*} \forall\ x\in B^*,\qquad
\frac{\|f(x)-(a+Tx)\|_{\ell_2}}{\rho}\le \e \|f\|_{\Lip}.
\end{equation}

Thus, all the Hilbert space-valued  $1$-Lipschitz functions on the
Euclidean unit ball of $\R^n$ are guaranteed to be $\e$-close to
some affine function on some sub-ball of radius at least $r_n(\e)$,
where $\e$-closeness is measured relative to the scale of the
sub-ball. A lower bound on $r_n(\e)$ corresponds to a
differentiation-type theorem asserting that any such function is  macroscopically close to being affine rather than being
infinitesimally affine. Crucially, the macroscopic lower bound on
the scale of affine approximability is independent of the given
function.

The basic question in which we are interested is that of determining
the asymptotic behavior of $r_n(\e)$ as $n\to \infty$.
Qualitatively, we ask for an asymptotic understanding of those
Hilbert space-valued Lipschitz functions on $B_{\ell_2^n}$ that are
hardest to approximate by affine functions.

There is a big gap between the known upper and lower bounds on
$r_n(\e)$. We have $r_n(\e)\le
e^{-cn/\e^2}$, where $c\in (0,\infty)$ is a universal
constant; see Section~\ref{sec:upper bounds} below.
The only known lower bound~\cite{LN12} on $r_n(\e)$ is $r_n(\e)\ge
e^{-(n/\e)^{Cn}}$, where $C\in (0,\infty)$ is a universal constant.
For concreteness, by choosing, say, $\e=\frac14$ and denoting
$r_n=r_n(\frac14)$, the best known bounds on $r_n$ become
\begin{equation}\label{eq:r_n bounds}
e^{-n^{Kn}}\le r_n\le e^{-\kappa n},
\end{equation}
for some universal constants $\kappa,K\in (0,\infty)$. An illustrative special case of the main result that is obtained here (to be described in full below) is the following asymptotic
improvement over the leftmost inequality~\eqref{eq:r_n bounds}, which holds for every $n\in \N$ and for some universal constant $K\in (0,\infty)$.
\begin{equation}\label{eq:modest}
r_n\ge e^{-e^{Kn}}.
\end{equation}

\subsection{The modulus of $L_p$ affine approximabilty}

Despite the fact that the above question was phrased in the context
of Hilbert spaces, a setting which arguably best highlights its
fundamental nature, it is important to study it in the context of
mappings between more general normed spaces; it is in this setting,
for example, that it becomes relevant to Bourgain's discretization
problem~\cite{Bou87,GNS11}, as explained in~\cite[Section~1.1]{LN12} (see Remark~\ref{rem:bourgain} below).

\begin{definition}\label{def Lp modulus} Fix $n\in \N$ and let $(X=\R^n,\|\cdot\|_X)$ be an $n$-dimensional
Banach space. Also, let $(Y,\|\cdot\|_Y)$ be an inifinite
dimensional Banach space. For $p\in (0,\infty]$ and $\e\in (0,1)$
define $r_p^{X\to Y}(\e)$ to be the supremum over those
$r\ge 0$ with the following property. For every Lipschitz function
$f:B_X\to Y$ there exists $y\in X$ and $\rho\in [r,\infty)$ such
that $y+\rho B_X\subseteq B_X$, and there exists $a\in Y$ and a
linear operator $T:X\to Y$ whose operator norm satisfies
$\|T\|_{X\to Y}\le 3\|f\|_{\Lip}$, such that
\begin{equation}\label{eq:def rp}
\left(\frac{1}{\vol(x+\rho B_X)}\int_{y+\rho B_X}\left\|f(z)-(a +Tz)\right\|_Y^p\dd z\right)^{\frac{1}{p}}\le
\e\rho\|f\|_{\Lip}.
\end{equation}
We call $r_p^{X\to Y}(\cdot)$ the {modulus of $L_p$  affine
approximability} corresponding to $X$ and $Y$.
\end{definition}

Using
the notation of Definition~\ref{def Lp modulus}, the quantity $r_n(\e)$ that we defined above can be written
as $$r_n(\e)\eqdef r_\infty^{\ell_2^n\to \ell_2}(\e).$$ Indeed,  the $L_\infty$ requirement~\eqref{UAAP def intro B^*}
implies that $\|T\|_{\ell_2^n\to \ell_2}\le (1+2\e)\|f\|_{\Lip}\le
3\|f\|_{\Lip}$. For finite $p$, the $L_p$ bound~\eqref{eq:def rp}
does not automatically imply a bound on $\|T\|_{X\to Y}$, which is
the reason why we added the requirement $\|T\|_{X\to Y}\le
3\|f\|_{\Lip}$ as part of the definition of $r_p^{X\to Y}(\e)$.


The case $p=\infty$ of Definition~\ref{def Lp modulus}, for which we shall use
below the simpler notation $$r^{X\to Y}(\e)\eqdef r_\infty^{X\to Y}(\e),$$
was introduced by Bates, Johnson, Lindenstrauss, Preiss and
Schechtman~\cite{BaJoLi}, who proved that $r^{X\to Y}(\e)>0$ for all
$\e\in (0,1)$ if and only if $Y$ admits an equivalent uniformly
convex norm; see~\cite{BaJoLi} for beautiful geometric applications of this result. The best known lower bound on $r^{X\to Y}(\e)$ (in
terms of $n$, $\e$ and the modulus of uniform convexity of $Y$) was
obtained in~\cite{LN12}. This bound is
\begin{equation}\label{LN uniformly convex}
r^{X\to Y}(\e)\ge e^{-(n/\e)^{c(Y)n}},
\end{equation}
where $c(Y)\in (0,\infty)$ depends only the modulus of uniform
convexity of the target Banach space $(Y,\|\cdot\|_Y)$. An explicit
estimate on $c(Y)$ appears in~\cite[Thm.~1.1]{LN12}. Here we obtain
an estimate that is asymptotically better than~\eqref{LN uniformly
convex} as $n\to \infty$ provided that $Y$-valued martingale differences are
unconditional ($Y$ is a UMD Banach space). Note that all the classical
reflexive Banach spaces have this property, but one can
construct~\cite{Pis75} uniformly convex Banach spaces that are not UMD.

Formally, a Banach space $(Y,\|\cdot\|_Y)$ is said to be a UMD
Banach space if there exists $\beta\in (1,\infty)$ such that if
$\{M_j\}_{j=0}^\infty$ is a $Y$-valued square-integrable martingale
defined on some probability space $(\Omega,\P)$ then for every $n\in
\N$ and every choice of signs $\e_1,\ldots,\e_n\in \{-1,1\}$ we have
\begin{equation}\label{eq:def UMD}
\int_\Omega \Big\|M_0+\sum_{j=1}^n \e_n\left(M_j-M_{j-1}\right)\Big\|_Y^2\dd\P\le
\beta^2\int_\Omega \left\|M_n\right\|_Y^2\dd\P.
\end{equation}
If $Y$ is a UMD Banach space then the infimum over those $\beta\in
(1,\infty)$ for which~\eqref{eq:def UMD} holds true for every
square-integrable $Y$-valued martingale $\{M_j\}_{j=0}^\infty$  is denoted below by
$\beta(Y)$. Examples of UMD Banach spaces include all $L_p(\mu)$
spaces when $p\in (1,\infty)$, in which case $\beta(L_p(\mu))\asymp
p^2/(p-1)$. See~\cite{Burkholder} and the references therein for
more information on UMD spaces.

Theorem~\ref{thm:UAAP intro} below asserts an improved lower bound
on the modulus of affine approximability $r^{X\to Y}(\e)$, provided
that $Y$ is a UMD space.

\begin{theorem}\label{thm:UAAP intro}
There is a universal constant $c\in [1,\infty)$ such that for
every $n\in \N$ and $\beta\in [2,\infty)$, if $(X,\|\cdot\|_X)$ is an
$n$-dimensional normed space and $(Y,\|\cdot\|_Y)$ is a UMD Banach space
satisfying $\beta(Y)\le \beta$ then for every $\e\in (0,1/2)$ we have
$$
r^{X\to Y}(\e)\ge \exp\left(-\frac{(\beta n)^{c\beta}}{\e^{c(n+\beta)}}\right).
$$
\end{theorem}

\begin{remark}\label{rem:bourgain}
{\em By substituting Theorem~\ref{thm:UAAP intro}  into equation (12) of~\cite{LN12} one obtains a bound on Bourgain's discretization modulus in the special case of UMD targets that  improves over the bound that was deduced in~\cite[Section~1.1]{LN12} and matches Bourgain's original bound~\cite{Bou87}. Specifically, one obtains the refined estimate that appears in equation (2) of~\cite{GNS11}. This yields a new proof of the best known general bound  in Bourgain's discretization problem via an approach that is entirely different from Bourgain's method, albeit in the special (though still very general) case of UMD targets. We note that due to the recent progress in~\cite{GNS11} a stronger bound is available here when the target is $L_p$.}
\end{remark}

The main reason why we study here the modulus of $L_p$ affine
approximability $r_p^{X\to Y}(\e)$ is that it relates to $r^{X\to
Y}(\e)$ through Lemma~\ref{lem:simple relation} below, whose simple
proof appears in Section~\ref{sec:infinity lemma}. The advantage
of working with finite $p$ is that it allows us to use a variety of
analytic tools, such as vector-valued Littlewood--Paley theory and complex
 interpolation.

\begin{lemma}\label{lem:simple relation}
Fix $n\in \N$ and $p\in [1,\infty)$. Suppose that $(X,\|\cdot\|_X)$
and $(Y,\|\cdot\|_Y)$ are Banach spaces with $\dim(X)=n$. Then for
every $\e\in (0,1)$ we have
$$
r^{X\to Y}(\e)\ge r_p^{X\to Y}\left(\left(\frac{\e}{9}\right)^{1+\frac{n}{p}}\right).
$$
\end{lemma}

Due to Lemma~\ref{lem:simple relation}, Theorem~\ref{thm:UAAP intro}
is a consequence of Theorem~\ref{thm:Lp UAAP intro} below, which is
our main result. Its proof is based on a vector-valued variant of an
argument of Dorronsoro~\cite{Do}, combined with a wide variety of
additional analytic and geometric ingredients of independent interest.

\begin{theorem}\label{thm:Lp UAAP intro}
There exist universal constants $\kappa,C\in [1,\infty)$ such that
for every $\beta\in [2,\infty)$ and $n\in \N$, if $(X,\|\cdot\|_X)$ is an
$n$-dimensional normed space and $(Y,\|\cdot\|_Y)$ is a UMD Banach space
satisfying $\beta(Y)\le \beta$, then for every $\e\in (0,1)$ we have
$$
r_{\kappa \beta }^{X\to Y}(\e)\ge \exp\left(-\frac{(\beta n)^{C\kappa\beta}}{\e^{\kappa\beta}}\right).
$$
\end{theorem}

\subsection{Previous work}\label{sec:previous} Over the past several decades, research on quantitative differentiation has proceeded roughly along two lines of inquiry, one of which arising from functional analysis and metric geometry and the other arising from rectifiability questions in harmonic analysis. The present work belongs to the former direction, but its main contribution is the use of methods from the latter direction in this new context while incorporating various Banach space theoretic tools.

Bates, Johnson, Lindenstrauss, Preiss and Schechtman studied~\cite{BaJoLi} quantitative differentiation in order to prove the rigidity of certain classes of Banach spaces under nonlinear quotients. The same notion was used in~\cite{LN12} for metric embeddings, namely as an alternative approach to Bourgain's discretization problem~\cite{Bou87} when the target Banach space is uniformly convex. The methods in this context fall under the category of (extensions of) ``approximate midpoint arguments," as initiated by Enflo~\cite{Ben85} to prove that $L_1$ is not uniformly homeomorphic to $\ell_1$, and further developed in~\cite{Bou87,EFW12,JLS96,Nao98} (see also Chapter 10 of~\cite{BL00}). As examples of the many related applications to quantitative embedding theory and rigidity questions, see also~\cite{Mat99,LR10,CKN11,EFW12,Che12,MN13,Li14,Li15}.

Parallel developments of a different nature arose in harmonic analysis, as part of the quest to develop a quantitative theory of rectifiability, with applications to singular integrals. Notable contributions along these lines include classical works of Stein (see the monograph~\cite{Ste70-singular}), through the works of Dorronsoro~\cite{Do}, Jones~\cite{Jon89,Jon90}, David and Semmes~\cite{DS91,DS93,DS00}, as well as the more recent work of Azzam and Schul~\cite{AS14}. The work of Dorronsoro~\cite{Do} directly influenced the present article, and we were also greatly inspired by the works of David and Semmes~\cite{DS91,DS93,DS00}. These works introduced and studied quantities that correspond to Definition~\ref{def Lp modulus} when $X$ is the Euclidean space $\R^n$ and $Y$ is the real line. Such methods also yield results for mappings from $\R^n$ to $\R^m$, but with statements that include implicit parameters that are allowed to depend on $m,n$. In~\cite{DS00} David and Semmes compare their work to that of Bates, Johnson, Lindenstrauss, Preiss and Schechtman~\cite{BaJoLi}, noting that the latter methods are different, and even yield results for infinite dimensional spaces.


Our contribution here follows the  harmonic-analytic   methodology, while overcoming several difficulties. Firstly, the literature on quantitative rectifiability ignores the dependence on the dimension $n$, while this dependence is the main topic of interest in the present context. In fact, a direct examination of the dependence on $n$ that is implicit in the above cited works reveals that it is insufficient for the purpose of obtaining improved bounds on $r^{X\to Y}(\e)$. Secondly, in our setting the  domain $X$ is a general $n$-dimensional normed space $X$ rather than a Euclidean space, and our arguments address this point. A final important difference is that we treat infinite dimensional Banach spaces $Y$ as targets. Overcoming this requires substantial effort, because the infinite dimensional arguments of Bates, Johnson, Lindenstrauss, Preiss and Schechtman~\cite{BaJoLi} do not seem to be applicable in our setting. The present work yields an infinite dimensional version of an inequality that Dorronsoro obtained~\cite{Do}  for real-valued functions. Such an infinite dimensional extension of Dorronsoro's work   is not routine, and in particular it does not hold true for arbitrary infinite dimensional Banach space targets $Y$. In fact, the geometry of $Y$ influences the structure of the inequality thus obtained, while stronger inequalities hold true for real-valued functions. The assumption that $Y$ is UMD is used several times in our argument through a rich UMD-valued Fourier-analytic toolkit that has been developed by many authors over the past four decades.


\subsection{Open questions}\label{sec:open} We list below some open questions that arise naturally from our work.

\begin{question}[Asymptotics of $r^{X\to Y}(\e)$]\label{Q:UAAP} {\em Obviously, the most tantalizing open question in the present context is to determine the rate at which $r_n$ tends to $0$ as $n\to \infty$, even roughly: say, is  this rate exponential, doubly exponential, or of some intermediate behavior? More importantly for potential applications, it remains open to obtain sharp bounds on the quantity $r^{X\to Y}(\e)$ when $X$ is a finite dimensional Banach space, $Y$ is a uniformly convex Banach space or belongs to some important class of Banach spaces (e.g. UMD spaces or uniformly convex lattices), and $\e\in (0,1)$.}
\end{question}

    \begin{question}[Infinite dimensional domains]{\em The question of characterizing those pairs of Banach space $X,Y$ for which $r^{X\to Y}(\e)>0$ for every $\e\in (0,1)$ was solved in~\cite{BaJoLi} also when $\dim(X)=\infty$ and $\dim(Y)<\infty$: this happens if and only if $X$  admits an equivalent uniformly smooth norm.
    It remains open to understand the asymptotic behavior of $r^{X\to Y}(\e)$ in the setting of uniformly smooth infinite dimensional domains and finite dimensional ranges. In particular, the rate at which $r^{\ell_2\to\ell_2^n}(\frac14)$ tends to $0$ as $n\to \infty$ is unknown. An explicit lower bound on  $r^{X\to Y}(\e)$ in terms of the modulus of uniform smoothness of $X$ and $\dim(Y)$ can be deduced from an examination of the proof in~\cite{BaJoLi}, but we believe that this bound is far from being optimal.}
    \end{question}

    \begin{question}[Uniformly convex targets]\label{Q:UC} {\em As stated earlier, by~\cite{BaJoLi} we know that if $\dim(X)=n<\infty$  and $\dim(Y)=\infty$ then $r^{X\to Y}(\e)>0$ for every $\e\in (0,1)$ if and only if $Y$ admits and equivalent uniformly convex norm.  However, the best known lower bounds on $r^{X\to Y}(\e)$ in this (maximal) generality remain those of~\cite{LN12}, and these bounds are weaker than Theorem~\ref{thm:UAAP intro} (as $n\to \infty$), which does not cover all uniformly convex ranges $Y$ because Pisier~\cite{Pis75} proved that there exist uniformly convex Banach spaces that are not UMD (even such uniformly convex Banach lattices exist, as shown by Bourgain in~\cite{Bou83}; see also the recent example by Qiu~\cite{Qiu}). It would be interesting to obtain an improved bound as in Theorem~\ref{thm:UAAP intro} under the weaker assumption that $Y$ is uniformly convex. Our proof of Theorem~\ref{thm:UAAP intro} definitely uses properties of $Y$ that imply the UMD property (e.g., we rely on the boundedness  of the $Y$-valued Hilbert transform, which was shown by Bourgain~\cite{Bou83} to imply that $Y$ is UMD). However, the conclusion of Theorem~\ref{thm:UAAP intro}, or even Theorem~\ref{thm:Lp UAAP intro}, may be valid when $Y$ is uniformly convex, and the same holds true for some of our other results, such as Theorem~\ref{w1p dorronsoro}  below. Certain aspects of vector-valued Littelwood--Paley theory are known to hold true for uniformly convex targets (see e.g.~\cite{MTX06}), so it would be interesting to investigate the extent to which the UMD property is needed for our results. If, on the other hand,
        Theorem~\ref{thm:Lp UAAP intro} or Theorem~\ref{w1p dorronsoro}  imply the UMD property then this would be a new characterization of UMD spaces.}
        \end{question}

\begin{question}[Asymptotics of $r^{X\to Y}_p(\e)$ for finite $p$] {\em  It would be interesting to understand the asymptotic behavior of the modulus of $L_p$ affine approximability $r_p^{X\to Y}(\e)$, even in the special case when $X$ and $Y$ are both Hilbert spaces, $p=2$ and, say, $\e=1/2$. A careful examination of the proof of Theorem~\ref{thm:Lp UAAP intro} in this Hilbertian setting  (in which case some of the steps that we perform below are not needed, and several estimates can be easily improved) reveals that for some $c\in (0,\infty)$,
\begin{equation}\label{eq:n4}
\forall\, n\in \N,\ \forall\,\e\in (0,1),\qquad  r_2^{\ell_2^n\to \ell_2}(\e)\gtrsim e^{-\frac{c(n\log n)^2}{\e^2}}.
\end{equation}
We do not know the extent to which~\eqref{eq:n4} is best possible; it seems plausible that with more work one could improve the dependence on $n$ in the exponent, but we do not presently have an upper bound that comes close to the lower bound in~\eqref{eq:n4}. Note that while the moduli  $r_p^{X\to Y}(\e)$ are interesting in their own right, we do not have geometric applications of them as in the case $p=\infty$. So, as a more amorphous research direction, it would be interesting to find geometric applications of bounds on $r_p^{X\to Y}(\e)$ (other than as a tool to bound $r_\infty^{X\to Y}(\e)$, which is the application that we present  here).}
\end{question}

\section{Geometric preliminaries}\label{sec:prelim}

Fix from now on an integer $n$ and an $n$-dimensional normed space
$(X,\|\cdot\|_X)$. We shall also fix a normed space
$(Y,\|\cdot\|_Y)$. In later sections we will need $Y$ to be a UMD
Banach space, but the statements of the present section hold true when
$Y$ is a general normed space.

By John's theorem~\cite{John} there exists a scalar product $\langle \cdot,\cdot\rangle$ on $X$ with respect to which we can identify $X$ with $\R^n$ and we have
\begin{equation}\label{eq:john posiiton}
\forall\, x\in \R^n,\qquad \|x\|_{2}\le \|x\|_X\le \sqrt{n}\cdot \|x\|_{2},
\end{equation}
where $\|\cdot\|_2=\|\cdot\|_{\ell_2^n}$. We shall also use the standard notation $B^n\eqdef B_{\ell_2^n}$ and $S^{n-1}\eqdef \partial B^n$. This Euclidean structure will be fixed from now on. Despite the fact that $X$ is now endowed with two metrics (those induced by $\|\cdot\|_X$ and $\|\cdot\|_2$), we shall tacitly maintain throughout the ensuing discussion the convention that whenever $\Omega\subset \R^n$ and  $f:\Omega\to Y$ then $\|f\|_{\Lip(\Omega)}$ denotes the Lipschitz constant of $f$ with respect to the metric induced by $\|\cdot\|_X$, i.e.,
$$
\|f\|_{\Lip(\Omega)}\eqdef \sup_{\substack{x,y\in \Omega\\ x\neq y}}\frac{\|f(x)-f(y)\|_Y}{\|x-y\|_X}.
$$
When $\Omega=\R^n$ we shall also use the shorter notation $\|f\|_{\Lip}=\|f\|_{\Lip(\R^n)}$.

We shall use standard notation for vector-valued $L_p$ spaces. Specifically, for every measurable subset $\Omega\subset \R^n$ of positive Lebesgue measure and $p\in [1,\infty]$, we let $L_p(\Omega,Y)$ denote the space of all measurable functions $f:\Omega\to Y$ such that
$$
\|f\|_{L_p(\Omega,Y)}\eqdef\left(\int_\Omega \|f(x)\|_Y^p\dd x\right)^{\frac{1}{p}}<\infty.
$$
Given $f:\Omega\to Y$ and $x\in \R^n$ we denote by $f^x:\Omega-x\to Y$ the translate of $f$ by $x$, i.e.,
\begin{equation}\label{eq:def translate}
\forall\, y\in \Omega-x,\qquad f^x(y)\eqdef f(x+y).
\end{equation}

For $u\in (0,\infty)$ and $f\in L_1(uB^n,Y)$ let $T_uf:\R^n\to Y$ be the linear operator defined by
\begin{equation}\label{eq:def T_u}
\forall\, w\in \R^n,\qquad T_uf(w)\eqdef \frac{n+2}{{V}_nu} \int_{B^n} \langle z,w\rangle f(uz)\dd z,
\end{equation}
where
\begin{equation*}\label{eq:def omega_n}
V_n\eqdef \vol(B^n)=\frac{\pi^{n/2}}{\Gamma\left(\frac{n}{2}+1\right)}.
\end{equation*}

The operator norm of $T_uf$ can be bounded in terms of the Lipschitz constant of $f$ as follows.
\begin{lemma}\label{lem Yu lip}
Fix $u\in (0,\infty)$ and a Lipschitz map $f:uB^n\to Y$. Then
$$
\|T_uf\|_{X\to Y}\le \|f\|_{\Lip(uB^n)}.
$$
\end{lemma}

\begin{proof}
By rescaling we may assume that $u=1$. Take $w\in \R^n\setminus \{0\}$. For every $t\in \R$ consider the following affine hyperplane.
$$
H_t\eqdef \left\{y\in \R^n:\ \left\langle y,\frac{w}{\|w\|_2}\right\rangle =t\right\}\subset \R^n.
$$
Then by the definition~\eqref{eq:def T_u} and Fubini's theorem we have
\begin{align*}
T_1f(w)&=\frac{(n+2)\|w\|_2}{V_n}\int_{-\infty}^\infty t\int_{H_t\cap B^n} f(z)\dd z \dd t\\ &=\frac{(n+2)\|w\|_2}{V_n}\int_{0}^\infty t\int_{H_t\cap B^n} \left(f(z)-f\left(z-\frac{2tw}{\|w\|_2}\right)\right)\dd z \dd t.
\end{align*}
Consequently,
\begin{align*}
\|T_1f(w)\|_Y&\le \frac{(n+2)\|w\|_2}{V_n}\int_{0}^\infty t\int_{H_t\cap B^n} \left\|f(z)-f\left(z-\frac{2tw}{\|w\|_2}\right)\right\|_Y\dd z\dd t\\
&\le \frac{2(n+2)\|f\|_{\Lip(B^n)}\|w\|_X}{V_n}\int_{0}^\infty t^2\vol_{n-1}\left(H_t\cap B^n\right)\dd t\\
&= \frac{(n+2)\|f\|_{\Lip(B^n)}\|w\|_X}{V_n}\int_{B^n} x_1^2\dd x\\
&= \|f\|_{\Lip(B^n)}\|w\|_X.\qedhere
\end{align*}
\end{proof}

We also record for future use the following simple estimate.

\begin{lemma}\label{lem:norm of T_u} For every $u\in (0,\infty)$ and $p\in [1,\infty]$ we have the following operator norm bound.
$$
\|T_u\|_{L_p(uB^n,Y)\to L_p(uB^n,Y)}\lesssim \min\left\{\sqrt{pn},n\right\}.
$$
\end{lemma}

\begin{proof}
By rescaling we may assume that $u=1$. If $f\in L_p(B^n,Y)$ then
\begin{align}
\label{eq:use triangle in LpY}\|T_1f\|_{L_p(B^n,Y)}&\le \frac{n+2}{V_n}\int_{B^n}\left\|w\mapsto \langle z,w\rangle f(z)\right\|_{L_p(B^n,Y)}\dd z\\\label{eq:use rotation invariance}&=
 \frac{n+2}{V_n}\bigg(\int_{B^n}|w_1|^pdw\bigg)^{\frac{1}{p}}\int_{B^n}\|z\|_2\cdot\|f(z)\|_Y\dd z\\\label{eq:compute p moment coordinate on sphere}
 &\lesssim \min\left\{\sqrt{pn},n\right\}\|f\|_{L_p(B^n,Y)},
\end{align}
where in~\eqref{eq:use triangle in LpY} we used the definition~\eqref{eq:def T_u} and the triangle inequality in $L_p(B^n,Y)$, in~\eqref{eq:use rotation invariance} we used rotation invariance, and~\eqref{eq:compute p moment coordinate on sphere} follows from H\"older's inequality combined with the following fact, which can be verified by a direct computation and is also, say, a special case of inequality (4) in~\cite{BGMN05}.
\begin{equation*}
\bigg(\frac{1}{V_n}\int_{B^n}|w_1|^p\dd w\bigg)^{\frac{1}{p}}\asymp \min\left\{\sqrt{\frac{p}{n}},1\right\}.\qedhere
\end{equation*}
\end{proof}

For $u\in (0,\infty)$ and $f\in L_1(uB^n,Y)$ define
\begin{equation}\label{eq:def P0}
P^0_uf\eqdef \frac{1}{V_n}\int_{B^n} f(uz)\dd z\in Y.
\end{equation}
Thus, for every $\Omega\subset \R^n$ and $f\in L_1(\Omega,Y)$, if $x\in \R^n$ and $u\in (0,\infty)$ satisfy $x+uB^n\subset \Omega$ then the vector $P^0_uf^x\in Y$ is the mean of $f$ over $x+uB^n$. The following simple estimate will be used later.

\begin{lemma}\label{lem:n+q}
Fix $p\in [1,\infty)$, $q\in (0,\infty)$ and $x\in \R^n$. Every measurable $f:\R^n\to Y$ satisfies
$$
\int_0^\infty \int_{B^n}\frac{\left\|f^x(uy)-P^0_uf^x\right\|_Y^p}{u^{q+1}}\dd y \dd u\le \frac{2^p}{n+q}\int_{\R^n} \frac{\|f^x(y)-f(x)\|_Y^p}{\|y\|_2^{n+q}}\dd y.
$$
\end{lemma}

\begin{proof}
Recalling~\eqref{eq:def P0}, it follows from the triangle inequality in $L_p(B^n,Y)$ and convexity that
$$
\forall\, v\in Y,\qquad \int_{B^n}\left\|f^x(uy)-P^0_uf^x\right\|_Y^p\dd y\le 2^p \int_{B^n}\left\|f^x(uy)-v\right\|_Y^p\dd y.
$$
By choosing here $v=f(x)$, we see that for every $u\in (0,\infty)$ we have
$$
\int_{B^n}\left\|f^x(uy)-P^0_uf^x\right\|_Y^p\dd y\le 2^p \int_{B^n}\left\|f^x(uy)-f(x)\right\|_Y^p\dd y.
$$
Hence, denoting the surface measure on $S^{n-1}$ by $\sigma$, by integrating in polar coordinates we get
\begin{align*}
\int_0^\infty \int_{B^n}\frac{\left\|f^x(uy)-P^0_uf^x\right\|_Y^p}{u^{q+1}}\dd y \dd u&\le 2^p\int_0^\infty\int_{B^n}\frac{\left\|f^x(uy)-f(x)\right\|_Y^p}{u^{q+1}}\dd y\dd u\\
&=2^p\int_0^\infty\int_0^1\int_{S^{n-1}}r^{n-1}\frac{\left\|f^x(urw)-f(x)\right\|_Y^p}{u^{q+1}}\dd\sigma(w)\dd r\dd u\\
&=2^p\int_0^1 \int_0^\infty \int_{S^{n-1}} r^{n-1}\frac{\left\|f^x(s w)-f(x)\right\|_Y^p}{(s/r)^{q+1}} \dd\sigma(w)\frac{\dd s}{r}\dd r\\
&=\frac{2^p}{n+q}\int_{\R^n} \frac{\|f^x(y)-f(x)\|_Y^p}{\|y\|_2^{n+q}}\dd y.\tag*{\qedhere}
\end{align*}
\end{proof}

Define an affine mapping $P^1_uf: \R^n\to Y$ by
\begin{equation}\label{eq:def P1}
P^1_uf\eqdef P^0_uf+T_uf.
\end{equation}
By a simple change of variable, for every $y\in \R^n$ we have
\begin{equation}\label{explicit legendre}
\left(P^1_{u} f^x\right)^{-x}(y)= P^0_{u}f^x+\sum_{j=1}^n\frac{y_j-x_j}{\int_{x+uB}(w_j-x_j)^2\dd w}\int_{x+uB}(z_j-x_j)f(z)\dd z.
\end{equation}
Consequently, if $f$ were a real-valued function then $(P^1_{u} f^x)^{-x}$ would be the orthogonal projection in $L_2(x+uB^n)$ of $f$ onto the subspace consisting of all the affine mappings.  Lemma~\ref{lem:projection closest in L_p} below shows that for every $p\in [1,\infty]$, if $f\in L_p(x+uB^n,Y)$ then the distance between $f$ and $(P^1_{u} f^x)^{-x}$ in  $L_p(x+uB^n,Y)$ is controlled by the distance of $f$ to the subspace of $L_p(x+uB^n,Y)$ consisting of all the affine mappings. Such a statement was previously proved in~\cite{Do}, but since  the dependence on $n$ and $p$ is important in the present context, and is only implicit in~\cite{Do}, we include its proof.

\begin{lemma}\label{lem:projection closest in L_p}
Fix $p\in [1,\infty]$, $u\in (0,\infty)$ and $x\in \R^n$. Suppose that $f\in L_p(x+uB^n,Y)$ and that $\Lambda:\R^n\to Y$ is affine. Then
$$
\left\|f^x-P^1_{u} f^x\right\|_{L_p(uB^n,Y)}=\left\|f-\left(P^1_{u} f^x\right)^{-x}\right\|_{L_p(x+uB^n,Y)}\lesssim \min\left\{\sqrt{pn},n\right\} \cdot \left\|f-\Lambda\right\|_{L_p(x+uB^n,Y)}.
$$
\end{lemma}

\begin{proof}
By  translation and a rescaling we may assume that $x=0$ and $u=1$. Since $P_1^1\Lambda=\Lambda$,
\begin{eqnarray*}
\left\|f-P^1_{1} f\right\|_{L_p(B^n,Y)}&\le& \left\|f-\Lambda\right\|_{L_p(B^n,Y)}+\left\|P^1_{1} (f-\Lambda)\right\|_{L_p(B^n,Y)}\\ &\stackrel{\eqref{eq:def P0}}{\le}&
\left\|f-\Lambda\right\|_{L_p(B^n,Y)}+\left\|P^0_{1} (f-\Lambda)\right\|_{L_p(B^n,Y)}+\left\|T_{1} (f-\Lambda)\right\|_{L_p(B^n,Y)}.
\end{eqnarray*}
It remains to note that $\|P^0_{1} (f-\Lambda)\|_{L_p(B^n,Y)}\le \|f-\Lambda\|_{L_p(B^n,Y)}$ since $P_1^0$ is an averaging operator, and to apply Lemma~\ref{lem:norm of T_u}.
\end{proof}

We end this section by recording for ease of later reference two consequences of Lemma~\ref{lem:projection closest in L_p}.
\begin{corollary}\label{lem:fractional sobolev appears}
Fix $p\in [1,\infty)$, $q\in (0,\infty)$ and $x\in \R^n$. Every measurable $f:\R^n\to Y$ satisfies
\begin{equation*}\label{eq:for use later bring in sobolev}
\left(\int_0^\infty\int_{B^n}\frac{\|f^x(uy)-P^1_uf^x(uy)\|_Y^p}{u^{q+1}}\dd y\dd u\right)^{\frac{1}{p}}\lesssim \frac{\min\left\{\sqrt{pn},n\right\}}{(q+n)^{\frac{1}{p}}}\bigg(\int_{\R^n}\frac{\|f^x(y)-f(x)\|_Y^p}{\|y\|_2^{n+q}}\dd y\bigg)^{\frac{1}{p}}.
\end{equation*}
\end{corollary}

\begin{proof}
Fix $u\in (0,\infty)$. Lemma~\ref{lem:projection closest in L_p} implies that every affine mapping $\Lambda:\R^n\to Y$ satisfies
\begin{equation}\label{eq:move to Bn}
\left(\int_{B^n} \left\|f^x(uy)-P_u^1f^x(uy)\right\|_Y^p\dd y\right)^{\frac1{p}}\lesssim \min\left\{\sqrt{pn},n\right\}\left(\int_{B^n} \left\|f^x(uy)-\Lambda(uy)\right\|_Y^p\dd y\right)^{\frac1{p}}.
\end{equation}
An application of~\eqref{eq:move to Bn} when $\Lambda$ is the constant $P^0_uf^x$ shows that for every $u\in (0,\infty)$ we have
$$
\left(\int_{B^n}\|f^x(uy)-P^1_uf^x(uy)\|_Y^p\dd y\right)^{\frac{1}{p}}\lesssim \min\left\{\sqrt{pn},n\right\}\left(\int_{B^n}\left\|f^x(uy)-P^0_uf^x\right\|_Y^p\dd y \right)^{\frac{1}{p}}.
$$
This implies the desired estimate due to Lemma~\ref{lem:n+q}.\end{proof}

\begin{corollary}\label{lem:fractional sobolev appears2}
Fix $p\in [1,\infty)$, $q\in (p,\infty)$ and $x\in \R^n$. Suppose that $f:\R^n\to Y$ is smooth. Then
\begin{multline}\label{eq:fj desired}
\left(\int_0^\infty\int_{B^n}\frac{\|f^x(uy)-P^1_uf^x(uy)\|_Y^p}{u^{q+1}}\dd y\dd u\right)^{\frac{1}{p}}\\
\lesssim \frac{p\min\left\{\sqrt{pn},n\right\}}{q(n+q-p)^{\frac{1}{p}}}\sum_{j=1}^n\bigg(\int_{\R^n}\left\|\frac{\partial f}{\partial x_j}(x+y)-\frac{\partial f}{\partial x_j}(x)\right\|_Y^p\frac{\dd y}{\|y\|_2^{n+q-p}}\bigg)^{\frac{1}{p}}.
\end{multline}
\end{corollary}
\begin{proof} Suppose that $x\in \R^n $ and $u\in (0,\infty)$. For every $j\in \n$ define
\begin{equation}\label{eq:def aj}
a_j(u,x)\eqdef\int_0^1 P^0_{su} f_j^x \dd s\in Y,\qquad\mathrm{where}\qquad   f_j\eqdef \frac{\partial f}{\partial x_j}.
\end{equation}
Also, define an affine function $\Lambda_{x,u}:\R^n\to Y$ by setting
\begin{equation}\label{eq:def Lambda choice}
\forall\, y\in \R^n,\qquad \Lambda_{x,u}(y)\eqdef f(x)+\sum_{j=1}^n y_ja_j(u,x).
\end{equation}
An application of~\eqref{eq:move to Bn} with $\Lambda=\Lambda_{x,u}$ shows that
\begin{multline}\label{eq:distance to Lambda x u}
\left(\int_0^\infty\int_{B^n}\frac{\|f^x(uy)-P^1_uf^x(uy)\|_Y^p}{u^{q+1}}\dd y\dd u\right)^{\frac{1}{p}}\\\lesssim \min\{\sqrt{pn},n\}\left(\int_0^\infty\int_{B^n}\frac{\left\|f^x(uy)-\Lambda_{x,u}(uy)\right\|_Y^p}{u^{q+1}}\dd y\dd u\right)^{\frac{1}{p}}.
\end{multline}
Observe that
\begin{align*}
f^x(uy)-\Lambda_{x,u}(uy)\stackrel{\eqref{eq:def Lambda choice}}{=} \int_0^1 \frac{d}{ds} f^x(suy)\dd s-\sum_{j=1}^n uy_ja_j(u,x)
\stackrel{\eqref{eq:def aj}}{=} \sum_{j=1}^n \int_0^1uy_j\big(f_j^x(suy)-P^0_{su}f_j^x\big) \dd s.
\end{align*}
By the triangle inequality in $L_p(B^n,Y)$, this implies that for every $u\in (0,\infty)$,
\begin{multline}\label{eq:for hardy}
\left(\int_{B^n}\left\|f^x(uy)-\Lambda_{x,u}(uy)\right\|_Y^p\dd y\right)^{\frac{1}{p}}\le \sum_{j=1}^n \int_0^1 u\bigg(\int_{B^n}|y_j|^p\left\|f_j^x(suy)-P^0_{su}f_j^x\right\|_Y^p\dd y\bigg)^{\frac{1}{p}}\dd s\\
\le \sum_{j=1}^n \int_0^u\bigg(\int_{B^n}\left\|f_j^x(ty)-P^0_{t}f_j^x\right\|_Y^p\dd y\bigg)^{\frac{1}{p}}\dd t=\sum_{j=1}^n \int_0^u h_j(t)\dd t,
\end{multline}
where we used the crude estimate $\max_{y\in B^n}\max_{j\in \n}|y_j|\le 1$ and introduced the notation
\begin{equation}\label{eq:def hj}
\forall\, (j,t)\in \n\times (0,\infty),\qquad h_j(t)\eqdef \bigg(\int_{B^n}\left\|f_j^x(ty)-P^0_{t}f_j^x\right\|_Y^p\dd y\bigg)^{\frac{1}{p}}.
\end{equation}
A combination of~\eqref{eq:for hardy} with the triangle inequality and Hardy's inequality~\cite[Section~A.4]{Ste70-singular} yields
\begin{multline}\label{hardy used}
\left(\int_0^\infty\int_{B^n}\frac{\left\|f^x(uy)-\Lambda_{x,u}(uy)\right\|_Y^p}{u^{q+1}}\dd y\dd u\right)^{\frac{1}{p}}\le \sum_{j=1}^n \bigg(\int_0^\infty\bigg(\int_0^uh_j(t)\dd t\bigg)^p
\frac{\dd u}{u^{q+1}}\bigg)^{\frac{1}{p}}
 \\\le \sum_{j=1}^n \frac{p}{q}\left(\int_0^\infty \frac{h_j(u)^p}{u^{q-p+1}} \dd u\right)^{\frac{1}{p}}\stackrel{\eqref{eq:def hj}}{=}\frac{p}{q}\sum_{j=1}^n \bigg(\int_0^\infty\int_{B^n}\left\|f_j^x(ty)-P^0_{t}f_j^x\right\|_Y^p \frac{\dd u}{u^{q-p+1}} \bigg)^{\frac{1}{p}}.
\end{multline}
By applying Lemma~\ref{lem:n+q} to each of the functions $\{h_j\}_{j=1}^n$ (with $q$ replaced by $q-p$), we conclude from~\eqref{hardy used} that the following estimate holds true.
\begin{equation*}
\left(\int_0^\infty\int_{B^n}\frac{\left\|f^x(uy)-\Lambda_{x,u}(uy)\right\|_Y^p}{u^{q+1}}\dd y\dd u\right)^{\frac{1}{p}} \lesssim \frac{p}{q(n+q-p)^{\frac{1}{p}}} \sum_{j=1}^n \bigg(\int_{\R^n}\frac{\|f_j^x(y)-f_j(x)\|_Y^p}{\|y\|_2^{n+q-p}}\dd y\bigg)^{\frac{1}{p}}.
\end{equation*}
Recalling the definition of $f_j$ in~\eqref{eq:def aj}, the desired estimate~\eqref{eq:fj desired} now follows from~\eqref{eq:distance to Lambda x u}.
\end{proof}

\subsection{Proof of Lemma~\ref{lem:simple relation}}\label{sec:infinity lemma}

Fix $\e\in (0,1)$ and denote $\d\eqdef(\e/9)^{1+n/p}$. Fix also
$0<r<r_p^{X\to Y}(\d)$ and a $1$-Lipschitz function $f:B_X\to Y$. By
the definition of $r_p^{X\to Y}(\d)$,  we can find $y\in X$ and
$\rho\ge r$ such that $y+\rho B_X\subseteq B_X$, and there exists an
affine mapping $\Lambda:X\to Y$ with $\|\Lambda\|_{\Lip}\le 3$, such
that
\begin{equation}\label{eq:delta approx}
\frac{1}{\vol(\rho B_X)}\int_{y+\rho B_X}\left\|f(u)-\Lambda(u)\right\|_Y^p\dd u\le
\d^p\rho^p=\frac{\e^{p+n}\rho^p}{9^{p+n}}.
\end{equation}
We claim that~\eqref{eq:delta approx}, combined with the fact that
$\|f-\Lambda\|_{\Lip}\le \|f\|_{\Lip}+\|\Lambda\|_{\Lip}\le 4$, yields
\begin{equation}\label{eq:contrapositive simple lemma}
\forall\, z\in y+\rho B_X,\qquad \left\|f(z)-\Lambda(z)\right\|_Y\le \e \rho,
\end{equation}
thus completing the proof of Lemma~\ref{lem:simple relation}.

Assume for the sake of obtaining a contradiction
that~\eqref{eq:contrapositive simple lemma} fails. Then there exists
$z\in y+\rho B_x$ such that $\left\|f(z)-\Lambda(z)\right\|_Y> \e
\rho$. Write $\lambda=\e/9$ and define $w=\lambda y+(1-\lambda)z$.
Observe  that we have $\|w-z\|_X=\lambda\|y-z\|_X\le\lambda\rho$. Supposing that
$u\in X$ satisfies $\|u-w\|_X\le \lambda\rho$, we therefore obtain
\begin{equation}\label{eq:inclusion lambda ball}
\|u-y\|_X=\left\|(u-w)-(1-\lambda)(y-z)\right\|_X\le
\|u-w\|+(1-\lambda)\|y-z\|_X \le \rho.
\end{equation}
Moreover, since $\|f-\Lambda\|_{\Lip}\le 4$ we have
\begin{multline}\label{eq:lower on lambda ball}
\|f(u)-\Lambda(u)\|_Y\ge
\|f(z)-\Lambda(z)\|_Y-4\|u-z\|_X\\>\e\rho-4\|u-w\|_X-4\|w-z\|_X\ge
\e\rho-8\lambda\rho=\lambda\rho.
\end{multline}
It follows from~\eqref{eq:inclusion lambda ball} that $w+\lambda\rho
B_X\subseteq y+\rho B_X$, and it follows from~\eqref{eq:lower on
lambda ball} that for $u\in w+\lambda\rho B_X$ the integrand in the
left hand side of~\eqref{eq:delta approx} is strictly larger than
$\lambda \rho$. Hence, \eqref{eq:delta approx} yields the following
contradiction.
\begin{equation*}
\lambda^{p+n}\rho^p=\frac{\e^{p+n}\rho^p}{9^{p+n}}> \frac{\vol(\lambda\rho B_X)}{\vol(\rho B_X)}\lambda^p\rho^p=\lambda^{p+n}\rho^p. \tag*{\qed}
\end{equation*}

\subsection{An upper bound on the modulus of affine approximability}\label{sec:upper bounds} The example that is constructed below was obtained in collaboration with Charles Fefferman; we thank him for agreeing that we include it here. A simple construction from~\cite{LN12} shows that if $p\in [2,\infty)$ and $\e\in (0,1)$ then
\begin{equation}\label{eq:previous bound on r infty}
r^{\ell_2^n\to \ell_2(\ell_p)}(\e)\lesssim \frac{e^{-1/(c\e)^p}}{\sqrt{n}},
\end{equation}
where $c\in (0,\infty)$ is a universal constant. We shall now show how the example of~\cite{LN12}  can be tensorized so as to yield an improved dependence on $n$, and we shall also briefly discuss the problem of bounding $r_q^{X\to Y}(\e)$ for finite $q\ge 1$. The following lemma is an $L_q$-variant of Lemma~4.1 of~\cite{LN12}.

\begin{lemma}\label{lem:sawtooth} There exists a universal constant $C\in (0,\infty)$ with the following property. For every $\e\in (0,1)$ and $p\in [1,\infty)$ there exists a $1$-Lipschitz function $f:\R\to \ell_p$ such that for every $q\in [1,\infty]$ and every affine mapping $\Lambda:\R\to \ell_p$, if $a,b\in \R$ satisfy $-1\le a<b\le 1$ then
\begin{equation}\label{eq:lower b-a}
  \frac{b-a}{2}\ge 4e^{-(C/\e)^p}\implies \left(\frac{1}{b-a}\int_a^b \left\|f(x)-\Lambda(x)\right\|_p^q\dd x\right)^{\frac1{q}}>\e\cdot \frac{b-a}{2}.
\end{equation}
Consequently,
$$
r_q^{\R\to \ell_p}(\e)\lesssim e^{-(C/\e)^p}.
$$
\end{lemma}

\begin{proof} Let $\f:\R\to \R$ be the piecewise affine (``sawtooth") function defined by $\f(2\Z)=\{0\}$ and $\f(1+2\Z)=\{1\}$. Fix $m\in \N$ that will be determined later. Denoting  the standard basis of $\ell_p^m$ by $\{e_j\}_{j=1}^m$, define $f:\R\to \ell_p^m$ by setting for every $x\in \R$,
\begin{equation}\label{eq:def lp sawtooth}
f(x)\eqdef \frac{1}{m^{1/p}}\sum_{k=1}^m \frac{\f\left(2^kx\right)}{2^k}e_k.
\end{equation}
Since $\f$ is $1$-Lipschitz, it follows from~\eqref{eq:def lp sawtooth} that also $f$ is $1$-Lipschitz.

Fix $a,b\in \R$ satisfying $-1\le a\le b\le 1$ and $b-a\ge 4/2^m$. There exists $k\in \{1,\ldots, m\}$ such that $4/2^k\le b-a\le 8/2^k$. Hence there is $j\in \{0,\ldots,2^{k-1}-1\}$ for which $[j/2^{k-1},(j+1)/2^{k-1}]\subset [a,b]$. Then, since $b-a\le 8/2^k$, for every affine mapping $\Lambda:\R\to \ell_p^m$ we have
\begin{equation}\label{eq:restrict to dyadic}
\frac{1}{b-a}\int_a^b \left\|f(x)-\Lambda(x)\right\|_p^q\dd x\gtrsim 2^k\int_{j/2^{k-1}}^{(j+1)/2^{k-1}} \left\|f(x)-\Lambda(x)\right\|_p^q\dd x.
\end{equation}
Writing $\Lambda=(\Lambda_1,\ldots,\Lambda_m)$, where $\Lambda_1,\ldots,\Lambda_m:\R\to \R$ are affine, it follows from~\eqref{eq:def lp sawtooth} and~\eqref{eq:restrict to dyadic} that
\begin{align}\label{eq:choose correct coordinate}
\nonumber\frac{1}{b-a}\int_a^b \left\|f(x)-\Lambda(x)\right\|_p^q\dd x&\gtrsim  2^k\int_{j/2^{k-1}}^{(j+1)/2^{k-1}} \left|\frac{\f\left(2^kx\right)}{m^{1/q}2^k}-\Lambda_k(x)\right|^q\dd x\\&=\frac{1}{m^{q/p}2^{qk}}\int_{-1}^1\left|\f(y)-\mathscr{L}(y)\right|^q\dd y,
\end{align}
where $\mathscr{L}:\R\to \R$ is the affine function given for every $y\in \R$ by $\mathscr{L}(y)\eqdef m^{1/p}2^k\Lambda_k((y+2j+1)/2^k)$. Recalling~\eqref{eq:def P0} and~\eqref{explicit legendre}, we have $P_1^1\f\equiv 1/2$. Hence, by Lemma~\ref{lem:projection closest in L_p},
\begin{equation}\label{eq:distance to legendre 1 dim}
\left(\int_{-1}^1\left|\f(y)-\mathscr{L}(y)\right|^q\dd y\right)^{\frac{1}{q}}\gtrsim \left(\int_{-1}^1 \left|\f(y)-\frac12\right|^q\dd y\right)^{\frac{1}{q}}\gtrsim 1.
\end{equation}
Since $b-a\asymp 2^{-k}$, by combining~\eqref{eq:choose correct coordinate} and~\eqref{eq:distance to legendre 1 dim}  we see that
\begin{equation}\label{eq:get lower b-a with m}
\left(\frac{1}{b-a}\int_a^b \left\|f(x)-\Lambda(x)\right\|_p^q\dd x\right)^{\frac{1}{q}}\ge \frac{\eta}{m^{1/p}}\cdot \frac{b-a}{2},
\end{equation}
where $\eta\in (0,\infty)$ is a universal constant. Suppose that $\e<\eta$ and choose $m$ to be the largest positive integer such that $m<(\eta/\e)^p$. Then~\eqref{eq:get lower b-a with m} implies the conclusion of~\eqref{eq:lower b-a}. The requirement here is that $b-a\ge 4/2^m$, and since $m+1\ge (\eta/\e)^p$, this requirement is satisfied if $b-a\ge 8/2^{(\eta/\e)^p}$, which follows from $(b-a)/2\ge 4e^{-(\eta/(2\e))^p}$. Thus~\eqref{eq:lower b-a} holds true with $C=\eta/2$. Note that, with this choice of $C$, the implication~\eqref{eq:lower b-a} holds vacuously when $\e\ge \eta$.
\end{proof}

Lemma~\ref{lem:simple tensorization} below tensorizes Lemma~\ref{lem:sawtooth} to improve over~\eqref{eq:previous bound on r infty}, obtaining exponential decay.

\begin{lemma}\label{lem:simple tensorization} There exists universal constants $K,\e_0\in (0,1)$ such that for every $\e\in (0,\e_0]$, every $p\in [1,\infty)$ and every $n\in \N$ we have
\begin{equation}\label{eq:exponential lower}
r^{\ell_2^n\to \ell_2(\ell_p)}(\e)\le e^{-n(K/\e)^p}.
\end{equation}
\end{lemma}

\begin{proof} By Lemma~\ref{lem:sawtooth} (with $q=\infty$) we can fix $K,\e_0\in (0,1)$ such that for every $\e\in (0,\e_0]$ and $p\in [1,\infty)$ there exists a $1$-Lipschitz mapping $f:[-1,1]\to \ell_p$ such that for every interval $[a,b]\subset [-1,1]$ with $(b-a)/2> e^{-(K/\e)^p}$ and for every affine mapping $\Lambda:\R\to \ell_p$ there exists $x\in [a,b]$ such that $\|f(x)-\Lambda(x)\|_p>\e(b-a)/2$.

Define $F:\R^n\to \ell_2^n(\ell_p)$ by setting for every $x=(x_1,\ldots,x_n)\in \R^n$,
$$
F(x)\eqdef \left(f\left(x_1\right),\frac{f\left(e^{(K/\e)^p}x_2\right)}{e^{(K/\e)^p}},\ldots,\frac{f\left(e^{(n-1)(K/\e)^p}x_n\right)}{e^{(n-1)(K/\e)^p}}\right)\in \ell_2^n(\ell_p).
$$
By applying the  $1$-Lipschitz condition for $f$ coordinate-wise, we see that $F$ is $1$-Lipschitz as a mapping from $\ell_2^n$ to $\ell_2^n(\ell_p)$. Fix $x\in B^n$ and $r\in (0,1)$ such that $x+rB^n\subset B^n$. Suppose from now on that $r>e^{-n(K/\e)^p}$ and let $j\in \N$ be the largest integer for which $r\le e^{-(j-1)(K/\e)^p}$. Then $j\in \n$ and $e^{(j-1)(K/\e)^p}r> e^{-(K/\e)^p}$. Let $\Lambda:\R^n\to \ell_2^n(\ell_p)$ be an affine mapping, and write $\Lambda=(\Lambda_1,\ldots,\Lambda_n)$, where $\Lambda_1,\ldots,\Lambda_n:\R^n\to \ell_p$ are affine. Consider the affine mapping $\Lambda_j':\R\to \ell_p$ given by $\Lambda_j'(y)=\Lambda_j(x_1,\ldots,x_{j-1},y,x_{j+1},\ldots,x_n)$.  Set $a=x_je^{(j-1)(K/\e)^p}-re^{(j-1)(K/\e)^p}$ and $b=x_je^{(j-1)(K/\e)^p}+re^{(j-1)(K/\e)^p}$. Then $$\frac{b-a}{2}=re^{(j-1)(K/\e)^p}>e^{-(K/\e)^p},$$ so by our assumption on $f$ there exists $w\in \left[x_je^{(j-1)(K/\e)^p}-re^{(j-1)(K/\e)^p}, x_je^{(j-1)(K/\e)^p}+re^{(j-1)(K/\e)^p}\right]$ with
$$
\left\|f(w)-e^{(j-1)(K/\e)^p}\Lambda_j'\left(we^{-(j-1)(K/\e)^p}\right)\right\|_p>\e r e^{(j-1)(K/\e)^p}.
$$
Setting $y=we^{-(j-1)(K/\e)^p}$, this is the same as asserting that there exists $y\in [x_j-r,x_j+r]$ with
$$
\left\|e^{-(j-1)(K/\e)^p}f\left(e^{(j-1)(K/\e)^p}y\right)-\Lambda'_j(y)\right\|_p>\e r.
$$
Hence, writing $z=(x_1,\ldots,x_{j-1},y,x_{j+1},\ldots,x_n)$, we have that $z\in x+rB^n$ and
\begin{equation}\label{eq:conclusion exponential}
\left\|F(z)-\Lambda(z)\right\|_{\ell_2^n(\ell_p)}\ge \left\|e^{-(j-1)(K/\e)^p}f\left(e^{(j-1)(K/\e)^p}y\right)-\Lambda'_j(y)\right\|_p>\e r,
\end{equation}
Since~\eqref{eq:conclusion exponential} holds true for every affine mapping $\Lambda:\R^n\to \ell_2^n(\ell_p)$ whenever $x+rB^n\subset B^n$ and $r>e^{-n(K/\e)^p}$, the proof of~\eqref{eq:exponential lower} is complete.
\end{proof}

\begin{remark}
{\em An upper bound on $r_p^{\ell_2^n\to \ell_2(\ell_p)}(\e)$ is a consequence of Lemma~\ref{lem:simple tensorization} and Lemma~\ref{lem:simple relation}. It seems likely that a significantly stronger upper bound holds true, but the above tensorization procedure does not seem to yield such an improvement. We leave the investigation of upper bounds on the modulus of $L_p$ affine approximability as an interesting question for future research.}
\end{remark}

\section{A UMD-valued Dorronsoro-type estimate}\label{sec:dorronsoro}

The main ingredient of the proof of Theorem~\ref{thm:Lp UAAP intro} (hence also Theorem~\ref{thm:UAAP intro}) is the following result.

\begin{theorem}\label{w1p dorronsoro} There exists a universal constant $\kappa\in [2,\infty)$ with the following property. Suppose that $(Y,\|\cdot\|_Y)$ is a UMD Banach space and write $\beta=\beta(Y)$. Suppose that $f:\R^n\to Y$ is a Lipschitz and compactly supported function. Then
\begin{equation}\label{eq:n factor}
\bigg(\frac{1}{V_n}\iiint_{ \R^n\times B^n\times (0,\infty)}\frac{\left\|f^x(uy)-P_u^1f^x(uy)\right\|_Y^{\kappa\beta}}{ u^{1+\kappa\beta}}\dd x  \dd y \dd u\bigg)^{\frac{1}{\kappa\beta}}
\lesssim \beta^{15}n^{\frac52}\big(\vol(\supp(f))\big)^{\frac{1}{\kappa\beta}}\|f\|_{\Lip}.
\end{equation}
\end{theorem}

In the case of real-valued functions on $\R^n$, such a statement  was first proved by Dorronsoro~\cite{Do}, with an implicit dependence on the dimension $n$.
Extensions and variants of Dorronsoro's theorem have been further studied within the theory of functions spaces, where norms like the one that appears on the left of~\eqref{eq:n factor} define what is called \emph{local approximation spaces}. See \cite[Section 1.7]{Triebel2} for some discussion of the subsequent history and \cite[Section 3.5.1]{Triebel2} for a different proof of a similar statement. There, the dependence on $n$ is perhaps even more implicit, since it also depends on a non-canonical choice of resolutions of identity used to define some general function spaces.

This dependence on dimension is crucial for us here; specifically, we desire polynomial growth in $n$ in the right hand side of~\eqref{eq:n factor}, while a na\"ive examination of the proof in~\cite{Do} reveals that it yields a much worse (super-exponential) dependence on $n$. Note also that in~\cite{Do}  the $L_{\kappa\beta}$ norm that appears in~\eqref{eq:n factor} can be replaced by an $L_p$ norm for any $p\in (1,\infty)$, while in the present vector-valued setting the geometry of the target space $Y$ influences the value of $p$. In fact, we shall prove a more refined (and stronger)  version of Theorem~\ref{w1p dorronsoro}; see Theorem~\ref{thM:doro more refined} below, in which the $L_{\kappa\beta}$ norm that appears in~\eqref{eq:n factor} can be replaced by an $L_p$ norm provided $Y$ is a UMD Banach space of Rademacher cotype $p$ (see Section~\ref{sec:cotype} below). These refinements are not important for our purposes, i.e., for proving Theorem~\ref{thm:Lp UAAP intro}, but they do imply sharper results, e.g. when $Y$ is itself an $L_q(\mu)$ space. In the same vein, the exponent $15$ of $\beta$ in~\eqref{eq:n factor} is not sharp (for the purpose of Theorem~\ref{thm:Lp UAAP intro} we only desire a polynomial dependence on $\beta$).

The bulk of the ensuing discussion is devoted to the proof of Theorem~\ref{w1p dorronsoro}. Our argument roughly follows the strategy of Dorronsoro in~\cite{Do}, combined with  substantial additions and modifications in order to obtain good dependence on $n$ and also overcome difficulties that arise in the vector-valued setting and are not present in the real-valued setting of~\cite{Do}. As explained in the Introduction, these complications reflect a genuine difference between the vector-valued setting and the real-valued setting, as such results do not hold true for general Banach space targets $Y$, so the geometry of $Y$ must somehow enter into the argument. We did not investigate the extent to which Theorem~\ref{w1p dorronsoro}  (and Theorem~\ref{thM:doro more refined} below) are sharp in terms of the assumptions that are required from $Y$  and the $L_p$ norm that could appear in~\eqref{eq:n factor} (or~\eqref{eq:doro refined} below).

\subsection{Proof of Theorem~\ref{thm:Lp UAAP intro} assuming the validity of Theorem~\ref{w1p dorronsoro}} Here we shall prove Theorem~\ref{thm:Lp UAAP intro} while  using Theorem~\ref{w1p dorronsoro}. This will allow us to focus later on Theorem~\ref{w1p dorronsoro} itself in order complete the justification of Theorem~\ref{thm:Lp UAAP intro}, and hence, by Lemma~\ref{lem:simple relation}, also~\eqref{eq:modest} and  Theorem~\ref{thm:UAAP intro}.

Recalling our setting, we are given a $n$-dimensional normed space $(X=\R^n,\|\cdot\|_X)$ such that~\eqref{eq:john posiiton} holds true. We are also given a normed space $(Y,\|\cdot\|_Y)$ that satisfies the assumptions of Theorem~\ref{w1p dorronsoro}. Suppose that $f:B_X\to Y$ is $1$-Lipschitz. Without loss of generality assume also that $f(0)=0$.

Define $\phi:[0,\infty)\to [0,\infty)$ by
$$
\forall\, u\in [0,\infty),\qquad \phi(u)\eqdef \left\{\begin{array}{ll}1&\mathrm{if}\ u\le \frac{1}{\sqrt{n}},\\
n+1-n^{\frac32}u & \mathrm{if}\ \frac{1}{\sqrt{n}}<u\le \left(1+\frac{1}{n}\right)\frac{1}{\sqrt{n}},\\
0&\mathrm{if}\ u>\left(1+\frac{1}{n}\right)\frac{1}{\sqrt{n}}.\end{array}\right.
$$
Then $\phi$ is supported in $[0,(1+1/n)/\sqrt{n}]$ and it is elementary to verify the following inequalities.
\begin{equation}\label{eq:phi upper}
\forall\, u\in (0,\infty),\qquad \max\big\{\phi(u),\sqrt{n}u\phi(u)\big\}\le 1.
\end{equation}
and
\begin{equation}\label{eq:phi lip}
\forall\, u,v\in (0,\infty),\qquad \big|\phi(u)-\phi(v)\big|\cdot\min\{u,v\}\le (n+1)|u-v|.
\end{equation}
Since by~\eqref{eq:john posiiton} we have $\frac{1}{\sqrt{n}}B^n\subset B_X$, by~\eqref{eq:phi upper} we know that every $x\in \R^n$ satisfies $\phi(\|x\|_2)x\in  B_X$. We can therefore define $F:\R^n\to Y$ by $F(x)=f(\phi(\|x\|_2)x)$. Then, $F(x)=f(x)$ if $\|x\|_2\le 1/\sqrt{n}$, and $F(x)=0$ if $\|x\|_2\ge (1+1/n)/\sqrt{n}$. Also, every $x,y\in \R^n$ with $\|y\|_X\le \|x\|_X$ satisfy
\begin{eqnarray*}
\nonumber\|F(x)-F(y)\|_Y&\le& \|f\|_{\Lip(B_X)}\big\|\phi(\|x\|_2)x-\phi(\|y\|_2)y\big\|_X\\&\le & \phi(\|x\|_2)\|x-y\|_X+\big|\phi(\|x\|_2)-\phi(\|y\|_2)\big|\|y\|_X\\&\stackrel{\eqref{eq:john posiiton} \wedge \eqref{eq:phi upper}}{\le}& \|x-y\|_X+
\big|\phi(\|x\|_2)-\|\phi(\|y\|_2)\big|\sqrt{n}\|y\|_2\\&\stackrel{\eqref{eq:phi lip}}{\le}&  \|x-y\|_X+\sqrt{n}(n+1)\|x-y\|_2\\&\stackrel{\eqref{eq:john posiiton}}{\le}& \left(n^{\frac32}+\sqrt{n}+1\right)\|x-y\|_X.
\end{eqnarray*}
Thus $\|F\|_{\Lip(\R^n)}\lesssim n^{\frac32}$. Since $\supp(F)\subset \left(1+\frac{1}{n}\right)\frac{1}{\sqrt{n}}B^n$, it therefore follows from Theorem~\ref{w1p dorronsoro} that
\begin{align}\label{eq:use doro F}
\nonumber\bigg(\iiint_{ \R^n\times B^n\times (0,\infty)}\frac{\left\|F^x(uy)-P_u^1F^x(uy)\right\|_Y^{\kappa\beta}}{V_n u^{1+\kappa\beta}}\dd x  \dd y \dd u\bigg)^{\frac{1}{\kappa\beta}}&\le K\beta^{15}n^{\frac52}\left(1+\frac{1}{n}\right)^{\frac{n}{\kappa\beta}}\left(\frac{V_n}{n^{\frac{n}{2}}}\right)^{\frac{1}{\kappa\beta}}n^{\frac32}\\&< eK \beta^{15}n^4\left(\frac{V_n}{n^{\frac{n}{2}}}\right)^{\frac{1}{\kappa\beta}},
\end{align}
where $K\in (0,\infty)$ is a universal constant. Note that
$$
(x,u)\in \left(\left(1-\frac{1}{n}\right)\frac{1}{\sqrt{n}}B^n\right)\times \left(0, \frac{1}{n^{3/2}}\right)\implies  x+uB^n\subset \frac{1}{\sqrt{n}} B^n.$$
Since $F$ coincides with $f$ on $\frac{1}{\sqrt{n}} B^n$, it follows that
$$
(x,y,u)\in \left(\left(1-\frac{1}{n}\right)\frac{1}{\sqrt{n}}B^n\right)\times B^n\times \left(0, \frac{1}{n^{3/2}}\right) \implies  P^1_uF^x=P^1_uf^x\quad\mathrm{and}\quad
 F^x(uy)=f^x(uy).$$
These observations in combination with~\eqref{eq:use doro F} imply that
\begin{equation}\label{eq:back to f}
\bigg(\iiint_{\left(\left(1-\frac{1}{n}\right)\frac{1}{\sqrt{n}}B^n\right)\times B^n\times \left(0, \frac{1}{n^{3/2}}\right)}\frac{\left\|f^x(uy)-P_u^1f^x(uy)\right\|_Y^{\kappa\beta}}{V_n u^{1+\kappa\beta}}\dd x  \dd y \dd u\bigg)^{\frac{1}{\kappa\beta}}< K\beta^{15}n^4\left(\frac{V_n}{n^{\frac{n}{2}}}\right)^{\frac{1}{\kappa\beta}}.
\end{equation}

Fix $M\in (1,\infty)$ whose precise value will be specified later. It follows from~\eqref{eq:back to f} that there exist
\begin{equation}\label{eq:u x choices}
u\in \left(\frac{1}{Mn^{2/3}},\frac1{n^{3/2}}\right)\qquad \mathrm{and}\qquad  x\in \left(1-\frac1{n}\right)\frac{1}{\sqrt{n}}B^n,
\end{equation}
such that
\begin{equation}\label{eq:log M}
\bigg(\frac{1}{V_n}\int_{B^n} \left\|f^x(uy)-P_u^1f^x(uy)\right\|_Y^{\kappa\beta}\dd y\bigg)^{\frac{1}{\kappa\beta}} \le \frac{2eK\beta^{15}n^4 u}{\left(\log M\right)^{\frac{1}{\kappa\beta}}}.
\end{equation}
Indeed, we would otherwise have
\begin{multline*}
\bigg(\iiint_{\left(\left(1-\frac{1}{n}\right)\frac{1}{\sqrt{n}}B^n\right)\times B^n\times \left(0, \frac{1}{n^{3/2}}\right)}\frac{\left\|f^x(uy)-P_u^1f^x(uy)\right\|_Y^{\kappa\beta}}{V_n u^{1+\kappa\beta}}\dd x  \dd y \dd u\bigg)^{\frac{1}{\kappa\beta}}\\> \left(1-\frac{1}{n}\right)^{\frac{n}{\kappa\beta}}\left(\frac{V_n}{n^{\frac{n}{2}}}\right)^{\frac{1}{\kappa\beta}}\left(\int_{\frac{n^{-3/2}}{M}}^{n^{-3/2}} \frac{\dd u}{u}\right)^{\frac{1}{\kappa\beta}}\frac{2eK\beta^{15}n^4 }{\left(\log M\right)^{\frac{1}{\kappa\beta}}}\ge eK\beta^{15}n^4\left(\frac{V_n}{n^{\frac{n}{2}}}\right)^{\frac{1}{\kappa\beta}},
\end{multline*}
thus contradicting~\eqref{eq:back to f}. Observe that by Lemma~\ref{lem Yu lip} we have
$$
\|P_u^1f^x\|_{\Lip(\R^n)}\stackrel{\eqref{eq:def P1}}{=}\|T_uf^x\|_{X\to Y}\le \|f^x\|_{\Lip(\R^n)}\le 1.
$$
Hence, if we set $\Lambda\eqdef P^1_xf^x$ then $\Lambda:\R^n\to Y$ is affine, $\|\Lambda\|_{\Lip(\R^n)}\le 1$, and  by a change of variable one can rewrite~\eqref{eq:log M} as follows.
\begin{equation}\label{eq:log M'}
\bigg(\frac{1}{\vol(x+uB^n)}\int_{x+uB^n} \left\|f(z)-\Lambda(z)\right\|_Y^{\kappa\beta}\dd z\bigg)^{\frac{1}{\kappa\beta}} \le \frac{2eK\beta^{15}n^4 u}{\left(\log M\right)^{\frac{1}{\kappa\beta}}}.
\end{equation}
This is not quite the type of conclusion that we desire, because the averaging in~\eqref{eq:log M'} occurs over a Euclidean ball rather than a ball in $(\R^n,\|\cdot\|_X)$. We overcome this via another averaging step.

Observe that
\begin{equation}\label{eq:mixed ball inclusion}
x+\left(1-\frac{1}{n}\right)uB^n+\frac{u}{n}B_X\subset x+uB^n\subset \frac{1}{\sqrt{n}}B^n\stackrel{\eqref{eq:john posiiton}}{\subset} B_X.
\end{equation}
Indeed, if $a,b\in \R^n$ satisfy $\|a\|_2\le (1-1/n)u$ and $\|b\|_X\le u/n$ then,
$$
\|a+b\|_2\le \|a\|_2+\|b\|_2\stackrel{\eqref{eq:john posiiton}}{\le} \|a\|_2+\|b\|_X\le \left(1-\frac{1}{n}\right)u+\frac{u}{n}=u,
$$
Now, define $A\subset \R^n\times \R^n$ as follows.
\begin{align*}
A&\eqdef\left\{(y,w):\ y\in x+\left(1-\frac{1}{n}\right)uB^n\ \wedge\ w\in y+\frac{u}{n}B_X\right\}\\
& =\left\{(y,w):\ w\in x+\left(1-\frac{1}{n}\right)uB^n+\frac{u}{n}B_X\ \wedge\ y\in
\left(x+\left(1-\frac{1}{n}\right)uB^n\right)\cap\left(w+\frac{u}{n}B_X\right)\right\}.
\end{align*}
By~\eqref{eq:mixed ball inclusion}, $f(w)$ is well-defined for every $(y,w)\in A$. It therefore follows from Fubini's theorem that
\begin{align*}
&\iint_A \|f(w)-\Lambda(w)\|_Y^{\kappa\beta}\dd w\dd y\\&=\int_{x+\left(1-\frac{1}{n}\right)u B^n} \dd y\int_{y+\frac{u}{n}B_X}\|f(w)-\Lambda(w)\|_Y^{\kappa\beta}\dd w\\&=
\int_{x+\left(1-\frac{1}{n}\right)u B^n+\frac{u}{n}B_X} \vol\left(\left(x+\left(1-\frac{1}{n}\right)uB^n\right)\cap\left(w+\frac{u}{n}B_X\right)\right)\|f(w)-\Lambda(w)\|_Y^{\kappa\beta}
\dd w\\&\le \vol\left(\frac{u}{n}B_X\right)\int_{x+uB^n} \|f(w)-\Lambda(w)\|_Y^{\kappa\beta}
\dd w.
\end{align*}
Hence,
\begin{multline*}
\frac{1}{\vol\left(x+\left(1-\frac{1}{n}\right)u B^n\right)}\int_{x+\left(1-\frac{1}{n}\right)u B^n} \frac{\dd y}{\vol\left(y+\frac{u}{n}B_X\right)}\int_{y+\frac{u}{n}B_X}\|f(w)-\Lambda(w)\|_Y^{\kappa\beta}\dd w\\
\le \frac{1}{\vol\left(x+\left(1-\frac{1}{n}\right)u B^n\right)} \int_{x+uB^n} \|f(w)-\Lambda(w)\|_Y^{\kappa\beta}
\dd w\le \frac{\left(2eK\beta^{15}n^4 u\right)^{\kappa\beta}}{\left(1-\frac{1}{n}\right)^n\log M}\le \frac{\left(4eK\beta^{15}n^4 u\right)^{\kappa\beta}}{\log M}.
\end{multline*}
This implies that there exists $y\in x+\left(1-\frac{1}{n}\right)u B^n$ such that
$$
\left( \frac{1}{\vol\left(y+\frac{u}{n}B_X\right)}\int_{y+\frac{u}{n}B_X}\|f(w)-\Lambda(w)\|_Y^{\kappa\beta}\dd w\right)^{\frac{1}{\kappa\beta}}\le \frac{4eK\beta^{15}n^4 u}{\left(\log M\right)^{\frac{1}{\kappa\beta}}}=\frac{4eK\beta^{15}n^5}{\left(\log M\right)^{\frac{1}{\kappa\beta}}}\cdot \frac{u}{n}.
$$
Recalling Definition~\eqref{def Lp modulus}, since by~\eqref{eq:u x choices} we are ensured that $u\ge 1/(Mn^{3/2})$, it follows that
\begin{equation}\label{eq:M version}
\forall\, M\in (1,\infty),\qquad r_{\kappa\beta}^{X\to Y}\left(\frac{4eK\beta^{15}n^5}{\left(\log M\right)^{\frac{1}{\kappa\beta}}}\right)\ge \frac{1}{Mn^{\frac32}}.
\end{equation}
For $\e\in (0,1)$ choose $M=e^{(4eK\beta^{15}n^5/\e)^{\kappa\beta}}$ in~\eqref{eq:M version}, thus yielding Theorem~\ref{thm:Lp UAAP intro} as follows.
\begin{equation*}
\forall\, \e\in (0,1),\qquad r_{\kappa\beta}^{X\to Y}(\e)\ge \frac{1}{n^{\frac32}}\exp\left(-\frac{(4eK\beta^{15}n^5)^{\kappa\beta}}{\e^{\kappa\beta}}\right).\tag*{\qed}
\end{equation*}

\section{Preliminaries on UMD Banach spaces}

This section is devoted to the presentation of several analytic properties of UMD Banach spaces that will be used extensively in the proof of Theorem~\ref{thm:Lp UAAP intro}.

Let $(Y,\|\cdot\|_Y)$ be a UMD Banach space and fix $p\in (1,\infty)$. Denote (as usual) by $\beta_p(Y)$ the infimum over those $\beta\in (1,\infty]$ such that  if
$\{M_j\}_{j=0}^\infty$ is a $Y$-valued $p$-integrable martingale
defined on some probability space $(\Omega,\P)$ then for every $n\in
\N$ and every $\e_1,\ldots,\e_n\in \{-1,1\}$ we have
\begin{equation}\label{eq:def UMDp}
\int_\Omega \Big\|M_0+\sum_{j=1}^n \e_j\left(M_j-M_{j-1}\right)\Big\|_Y^p\dd\P\le
\beta^p\int_\Omega \left\|M_n\right\|_Y^p\dd\P.
\end{equation}

Thus, using the notation of the Introduction, we have $\beta_2(Y)=\beta(Y)$. The following inequality is well-known; see~\cite{Bu86} for its proof.
\begin{equation}\label{eq:betap beta 2}
\beta_p(Y)\lesssim \frac{p^2}{p-1}\beta(Y).
\end{equation}
We also record for future use that~\eqref{eq:def UMDp} implies (see~\cite[Thm.~4.4]{LT91}) that for every $a_0,a_1,\ldots,a_n\in \R$,
\begin{equation}\label{eq:UMDp contraction}
\int_\Omega \Big\|a_0M_0+\sum_{j=1}^n a_j\left(M_j-M_{j-1}\right)\Big\|_Y^p\dd\P\le
\left(\max_{j\in \{0,\ldots,n\}}|a_j|^p\right)\beta_p(X)^p\int_\Omega \left\|M_n\right\|_Y^p\dd\P.
\end{equation}

In~\cite{Gar90} Garling introduced two parameters $\beta_p^+(Y), \beta_p^{-}(Y)$, defined to be the best constants in the following inequalities, which are required to hold true for every martingale $\{M_j\}_{j=0}^\infty$ as above.
\begin{equation*}\label{eq:def UMDp+}
\E_\e\bigg[\int_\Omega \Big\|M_0+\sum_{j=1}^n \e_n\left(M_j-M_{j-1}\right)\Big\|_Y^p\dd\P\bigg]\le
\beta_p^+(Y)^p\int_\Omega \left\|M_n\right\|_Y^p\dd\P,
\end{equation*}
and
\begin{equation*}\label{eq:def UMDp-}
\int_\Omega \left\|M_n\right\|_Y^p\dd\P\le \beta_p^-(Y)^p\E_\e\bigg[\int_\Omega \Big\|M_0+\sum_{j=1}^n \e_n\left(M_j-M_{j-1}\right)\Big\|_Y^p\dd\P\bigg],
\end{equation*}
where $\E_\e[\cdot]$ denotes the expectation with respect to $\e=(\e_1,\ldots,\e_n)$ chosen uniformly at random from $\{-1,1\}^n$. Garling's inequalities are weaker than~\eqref{eq:def UMDp}, which is required to hold for {\em every} $\e\in \{-1,1\}^n$ rather than only in expectation with respect to $\e$. Hence,
\begin{equation}\label{eq:trivial beta+ beta-}
\max\left\{\beta_p^+(Y),\beta_p^-(Y)\right\}\le \beta_p(Y),
\end{equation}
but there are examples of Banach spaces $Y$ for which $\beta_p^+(Y)$ or $\beta_p^-(Y)$ is markedly smaller than $\beta_p(Y)$; see~\cite{Gar90,Gei99}. Some of the ensuing estimates can be stated in terms of the parameters $\beta_p(Y), \beta_p^+(Y), \beta_p^-(Y)$, but in order to avoid cumbersome expressions we will sometimes state our bounds in terms of the quantity $\beta(Y):=\beta_2(Y)$, by invoking~\eqref{eq:betap beta 2} and \eqref{eq:trivial beta+ beta-}. In fact, in our setting we will always choose  $2\le p\asymp  \beta(Y)$, in which case by using~\eqref{eq:betap beta 2} and \eqref{eq:trivial beta+ beta-} we will sometimes bound from above the quantities $\beta_p(Y),\beta_p^+(Y), \beta_p^-(Y)$ by a universal constant multiple of $p\,\beta(Y)\asymp\beta(Y)^2$.
This choice has the advantage of simplifying some of the ensuing discussion, but it yields bounds that could be improved for some (quite exotic) UMD Banach spaces $Y$ by a straightforward inspection of our proofs.

\subsection{$\mathscr{R}$-boundedness} Let $(X,\|\cdot\|_X)$ and $(Y,\|\cdot\|_Y)$ be Banach spaces. The space of bounded linear operators from $X$ to $Y$ is denoted $\bddlin(X,Y)$. Following~\cite{BG:94,CPSW}, a set of operators $\mathcal{T}\subseteq\bddlin(X,Y)$ is said to be $\mathscr{R}$-bounded if
\begin{equation}\label{eq:def r bounded}
  \E_\e\bigg[\Big\|\sum_{j=1}^N\radem_j T_j x_j\Big\|_Y^p\bigg]
  \leq C^p\E_\e\bigg[\Big\|\sum_{j=1}^N\radem_j x_j\Big\|_X^p\bigg],
\end{equation}
for every $N\in\N$, every $x_1,\ldots, x_N\in Y$ and every $T_1,\ldots, T_N\in\mathcal{T}$, and some (equivalently, by Kahane's inequality~\cite{Kah85}, for all) $p\in[1,\infty)$. The infimum over those $C$ in~\eqref{eq:def r bounded} is denoted $\mathscr{R}_p(\mathcal{T})$.

The following result is due to Bourgain~\cite{Bourgain:86}; see also \cite{FigWoj} for a proof.
\begin{proposition}[Bourgain's vector-valued Stein inequality]\label{prop:Stein}
Suppose that $(Y,\|\cdot\|_Y)$ is a UMD Banach space and fix $p\in(1,\infty)$. Let $\{\mathscr{F}_j\}_{j\in \Z}$ be an increasing sequence of sub-$\sigma$-algebras on some probability space $(\Omega,\mathscr{F},\P)$. For every $j\in \Z$ let $\mathscr{E}_j\in \bddlin(L_p(\P,Y),L_p(\P,Y))$ be the conditional expectation operator corresponding to $\mathscr{F}_j$, i.e., $\mathscr{E}_jf=\E[f|\mathscr{F}_j]$ for every $f\in L_p(\P,Y)$. Then
$$
\mathscr{R}_p\left(\{\mathscr{E}_j\}_{j\in \Z}\right)\le \beta_p^+(Y).
$$
\end{proposition}

By a classical representation theorem for positive self-adjoint semigroups of contractions due to Rota~\cite{Rota} (see also~\cite[Sec.~VI]{Ste70} or~\cite[Thm.~2.5]{MTX06}), Proposition~\ref{prop:Stein} implies the following dimension independent $\mathscr{R}$-boundedness estimate for the heat semigroup. Here and in what follows, $\Delta$ denotes the Laplacian on $\R^n$.

\begin{corollary}[$\mathscr{R}$-boundedness of the heat semigroup]\label{cor:heat} Let $(Y,\|\cdot\|_Y)$ be a UMD Banach space, $n\in \N$, and $p\in (1,\infty)$. Then
$$
\mathscr{R}_p\left(\{e^{t\Delta}\}_{t\in (0,\infty)}\right)\le \beta_p^+(Y).
$$
\end{corollary}

\begin{proof}
For any $\delta>0$, the operator $Q=e^{\frac12\delta\Delta}$ satisfies the assumptions of Rota's theorem as formulated in \cite[Sec.~VI]{Ste70} or~\cite[Thm.~2.5]{MTX06}. Thus there exists a measure space $(\Omega,\mathscr F,\mu)$ with an increasing sequence of sub-$\sigma$-algebras $\{\mathscr{F}_j\}_{j\in \Z}$ and yet another sub-$\sigma$-algebra ${\mathscr F}'$ with the corresponding conditional expectation operators $\mathscr E_j$ and ${\mathscr E}'$ such that
\begin{equation*}
\forall\, j\in \N\cup\{0\},\qquad   e^{j\delta\Delta}=Q^{2j}=J^{-1}{\mathscr E}'\mathscr E_{-j}J,
\end{equation*}
where $J:L_p(\R^n,Y)\to L_p(\Omega,{\mathscr F}',Y)$ is an isometric isomorphism. Thus
\begin{align*}
  &\mathscr R_p\big(\{e^{j\delta\Delta}\}_{j\in\N\cup\{0\}}\big)\\
  &=\mathscr R_p\big(\{J{\mathscr E}'\mathscr E_{-j}J\}_{j\in\N\cup\{0\}}\big)
 \\ &\leq\|J^{-1}\|_{ L_p(\Omega,{\mathscr F}',Y)\to L_p(\R^n,Y)}\cdot \|{\mathscr E}'\|_{L_p(\Omega,\mathscr{F},\mu)\to L_p(\Omega,\mathscr{F'},\mu)}\cdot \mathscr R_p\big(\{\mathscr E_{-j}\}_{j\in\N\cup\{0\}}\big)\cdot \|J\|_{L_p(\R^n,Y)\to L_p(\Omega,{\mathscr F}',Y)}
 \\&\leq\beta_+(Y),
\end{align*}
where we used Proposition \ref{prop:Stein} together with easy properties of $\mathscr R$-bounds and the contractivity of conditional expectations.

Using basic properties of $\mathscr R$-bounds (cf. \cite[Proposition 9.5]{vNee07}), if $\mathscr T=\bigcup_{k=1}^\infty\mathscr T_k$ is a union of an increasing sequence $\mathscr T_k\subset\mathscr T_{k+1}$, and $\overline{\mathscr T}$ is the closure of $\mathscr T$ in the strong operator topology, then $\mathscr R_p(\overline{\mathscr T})=\mathscr R_p(\mathscr T)=\lim_{k\to\infty}\mathscr R_p(\mathscr T_k)$. With $\mathscr T_k=\{e^{j2^{-k}\Delta}\}_{j \in\N\cup \{0\}}$, we have $\overline{\mathscr T}=\{e^{t\Delta}\}_{t\in(0,\infty)}$ by the strong continuity of $t\mapsto e^{t\Delta}$, and this proves the claim.
\end{proof}

Our next goal is to prove Theorem~\ref{thm:A representation} below, which is a useful bound on the norm of  operators on UMD Banach spaces that admit a certain integral representation in terms of the heat semigroup.   Results in this spirit have been implicitly used for a long time; see for examples the probabilistic treatment of the Riesz transforms by Gundy and Varopoulos~\cite{GV:79}, and of the Beurling--Ahlfors transform by Ba\~nuelos and M\'endez-Hern\'andez~\cite{BM:03}. Formulations in the UMD-valued setting appear in \cite{Hytonen:aspects,Hytonen:LPS,Hytonen:BA}. The version below is essentially a combination of some of these earlier results but it does not appear as stated in the literature, so its proof is included here.

In what follows, we will use some aspects of the theory of vector-valued stochastic integration with respect to a Brownian motion $\{B(t)=(B_1(t),\ldots,B_n(t))\}_{t\in (0,\infty)}$ in $\R^n$, starting at $0$. Here and below, a Brownian motion in $\R^n$ is always understood to be a {\em standard} Brownian motion.


It should be noted that for
our purposes it is enough to consider finite-dimensional-valued
functions, in which case the stochastic integrals can be defined
coordinate-wise in the classical sense.
We refer to \cite[Chapter 3]{KS} for the relevant background (and much more).
It might be helpful to note that the formulae in \cite{KS}, which often involve the \emph{quadratic variation} $\langle M,N\rangle_t$ of two stochastic processes $M$ and $N$, take a simpler form in the Brownian case of our interest, by using the identities $\langle B_i,B_j\rangle_t=\delta_{ij}t$ (see \cite[Theorem 3.3.16]{KS}).

In particular, It\^o's formula (see \cite[Theorem 3.3.6]{KS}) is valid in our setting, since it holds true for each
scalar-valued coordinate function. For a comprehensive theory of vector-valued stochastic
integration, whose full strength is not needed here, see~\cite{NW05,NW07}.

Before stating Theorem~\ref{thm:A representation} we describe some preliminary background and simple estimates that will be used in its proof. First,  we recall the following decoupling inequalities due to Garling~\cite{Garling:86}. 

\begin{theorem}[Garling's decoupling inequalities]\label{thm:garling decoupling} For $n\in \N$, let $\{B(t)=(B_1(t),\ldots,B_n(t))\}_{t\in (0,\infty)}$ be a Brownian motion in $\R^n$, starting at $0$. Also, let $\{C(t)=(C_1(t),\ldots,C_n(t))\}_{t\in (0,\infty)}$ be an independent copy of $\{B(t)\}_{t\in (0,\infty)}$. Suppose that  $(Y,\|\cdot\|_Y)$ is a UMD Banach space and $p\in (1,\infty)$. Let $V=(V_1,\ldots,V_n):(0,\infty)\to Y^n$ be a stochastic process that is adapted to the same filtration as $\{B(t))\}_{t\in (0,\infty)}$, takes values in a finite-dimensional subspace of $Y^n$, and satisfies
\begin{equation}\label{eq:integrability hypothesis}
  \E\bigg[\Big(\int_0^\infty \sum_{j=1}^n \|V_j(t)\|_Y^2 \dd t\Big)^{\frac{p}{2}}\bigg]<\infty
\end{equation}
The finite dimensionality assumption and the integrability assumption~\eqref{eq:integrability hypothesis} guarantee the existence of the stochastic integrals below by the scalar-valued theory. Then
$$
\E\bigg[\Big\|\int_0^\infty \sum_{j=1}^n V_j(t)\dd B_j(t)\Big\|_Y^p\bigg]\le \beta_p^-(Y)^p \E\bigg[\Big\|\int_0^\infty \sum_{j=1}^n V_j(t)\dd C_j(t)\Big\|_Y^p\bigg],
$$
and
$$
\E\bigg[\Big\|\int_0^\infty \sum_{j=1}^n V_j(t)\dd C_j(t)\Big\|_Y^p\bigg]\le \beta_p^+(Y)^p\E\bigg[\Big\|\int_0^\infty \sum_{j=1}^n V_j(t)\dd B_j(t)\Big\|_Y^p\bigg].
$$
\end{theorem}

Continuing with the notation of Theorem~\ref{thm:garling decoupling}, for every (operator-valued) measurable\footnote{Here, and in what follows, given two Banach spaces $(X,\|\cdot\|_X)$ and $(Y,\|\cdot\|_Y)$ and an open subset $\Omega\subset \R^n$, when we say that an operator-valued mapping $\Phi:\Omega\to \mathscr{L}(X,Y)$ is measurable we mean measurability in the strong operator topology, i.e., we require that for every $x\in X$ the mapping $w\mapsto \Phi(w)x$ from $\Omega$ to $Y$ has the property that the inverse image of every Borel subset of $Y$ is Lebesgue-measurable.} mapping $\Phi:(0,\infty)\to \mathscr{L}(Y,Y)$ we have
\begin{equation}\label{eq:R-boundedness stochastic integral}
\E\bigg[\Big\|\int_0^\infty \sum_{j=1}^n \Phi(t)V_j(t)\dd C_j(t)\Big\|_Y^p\bigg]\le \mathscr{R}_p
\left(\Phi\right)^p\E\bigg[\Big\|\int_0^\infty \sum_{j=1}^n V_j(t)\dd C_j(t)\Big\|_Y^p\bigg],
\end{equation}
where we use the notation
\begin{equation}\label{eq:R_p(A) notation}
\mathscr{R}_p(\Phi)\eqdef \mathscr{R}_p
\left(\left\{\Phi(t)\right\}_{t\in (0,\infty)}\right).
\end{equation}
The estimate~\eqref{eq:R-boundedness stochastic integral} follows directly from the definition of $\mathscr{R}$-boundedness by approximating the integrals by Riemann sums; see Exercise~4 in Section~9.4 of~\cite{vNee07}. Alternatively, inequality~\eqref{eq:R-boundedness stochastic integral} follows by combining Theorem~6.14 and Theorem~9.13 of~\cite{vNee07}. By~\eqref{eq:R-boundedness stochastic integral} and Theorem~\ref{thm:garling decoupling} we see that
 \begin{equation}\label{eq:R-boundedness stochastic integral-coupled}
\E\bigg[\Big\|\int_0^\infty \sum_{j=1}^n \Phi(t)V_j(t)\dd B_j(t)\Big\|_Y^p\bigg]\le \beta_p^+(Y)^p\beta_p^-(Y)^p\mathscr{R}_p(\Phi)^p
\E\bigg[\Big\|\int_0^\infty \sum_{j=1}^n V_j(t)\dd B_j(t)\Big\|_Y^p\bigg].
\end{equation}

 In the same vein as the above discussion, by approximating the integrals by Riemann sums it follows from~\eqref{eq:UMDp contraction} that if  $\phi:(0,\infty)\to \R$ is measurable then
\begin{equation}\label{eq:contraction stochastic integral}
\E\bigg[\Big\|\int_0^\infty \sum_{j=1}^n \phi(t)V_j(t)\dd B_j(t)\Big\|_Y^p\bigg]\\\le \beta_p(Y)^p\|\phi\|_{L_\infty(0,\infty)}^p
\E\bigg[\Big\|\int_0^\infty \sum_{j=1}^n V_j(t)\dd B_j(t)\Big\|_Y^p\bigg].
\end{equation}

We record for future use the following simple estimate.

\begin{lemma}\label{lem:use ito}
Fix $n\in \N$ and let $\{B(t)=(B_1(t),\ldots,B_n(t))\}_{t\in [0,\infty)}$ be a Brownian motion in $\R^n$, starting at $0$. For every Banach space $(Y,\|\cdot\|_Y)$, every  $p\in (1,\infty)$ and every smooth and compactly supported $h:\R^n\to Y$ with a finite-dimensional range, we have
\begin{equation}\label{eq:limsup}
\limsup_{\tau\to \infty} \bigg(\int_{\R^n}\E\bigg[\Big\|\sum_{j=1}^n \int_0^\tau \frac{\partial}{\partial x_j} e^{\frac{\tau-t}{2}\Delta}h(B(t)+x)\dd B_j(t) \Big\|_Y^p\bigg]\dd x\bigg)^{\frac{1}{p}}\le \|h\|_{L_p(\R^n,Y)}.
\end{equation}
\end{lemma}

\begin{proof}
If  $\f\in L_1(\R^n)$ then $\E[\f(B(t)+x)]=e^{\frac{t}{2}\Delta}\f(x)$ for every $t\in [0,\infty)$ and $x\in \R^n$. Hence,
\begin{equation}\label{eq:markov property}
\int_{\R^n}\E[\f(B(t)+x)]\dd x=\int_{\R^n}e^{\frac{t}{2}\Delta}\f(x)\dd x=\int_{\R^n}\f(x)\dd x.
\end{equation}
For $(x,t)\in \R^n\times (0,\infty)$, let
$
u(x,t)\eqdef e^{\frac{t}{2}\Delta}h(x)
$
denote the heat extension of $h$. By It\^o's formula (see \cite[Theorem 3.3.6]{KS}) applied to the function $t\mapsto u(B(t)+x,\tau-t)$, for every $x\in \R^n$ we have
\begin{align}\label{eq:use ito}
\nonumber&\sum_{j=1}^n \int_0^\tau \frac{\partial u}{\partial x_j} (B(t)+x,\tau-t)\dd B_j(t)\\&=u(B(\tau)+x,0)-u(x,\tau)-\int_0^\tau \bigg(-\frac{\partial u}{\partial t}+\frac12 \Delta_x u\bigg)(B(t)+x,\tau-t) \dd t\\
&= h(B(\tau)+x)-e^{\frac{\tau}{2}\Delta}h(x).
\end{align}
Consequently,
\begin{align}
\nonumber&\bigg(\int_{\R^n}\E\bigg[\Big\|\sum_{j=1}^n \int_0^\tau \frac{\partial}{\partial x_j} e^{\frac{\tau-t}{2}\Delta}h(B(t)+x)\dd B_j(t) \Big\|_Y^p\bigg]\dd x\bigg)^{\frac{1}{p}}\\&\stackrel{\eqref{eq:use ito}}{\le} \left(\int_{\R^n}\E\left[\left\|h(B(\tau)+x)\right\|_Y^p\right]\dd x\right)^{\frac{1}{p}}+\left(\int_{\R^n}\E\left[\left\|e^{\frac{\tau}{2}\Delta}h(x)\right\|_Y^p\right]\dd x\right)^{\frac{1}{p}}\nonumber\\
&\stackrel{\eqref{eq:markov property}}{=} \|h\|_{L_p(\R^n,Y)}+\left\|e^{\frac{\tau}{2}\Delta}h\right\|_{L_p(\R^n,Y)}.\label{eq:before limsup}
\end{align}
To deduce the desired bound~\eqref{eq:limsup} from~\eqref{eq:before limsup}, it remains to note that
\begin{equation}\label{eq:heat limit at infty}
\forall\, p\in(1,\infty),\qquad \lim_{\tau\to \infty} \|e^{\frac{\tau}{2}\Delta}h\|_{L_p(\R^n,Y)}=0.
\end{equation}
 Indeed, by Young's inequality we have the following point-wise estimate
\begin{equation}\label{eq:heat kernel pointwise}
\forall\,x\in \R^n,\qquad \left\|e^{\frac{\tau}{2}\Delta}h(x)\right\|_Y=\left\| k_{\frac{\tau}{2}}*h(x)\right\|_Y\leq\left\| k_{\frac{\tau}{2}}\right\|_{L_q(\R^n)}\|h\|_{L_p(\R^n,Y)},
\end{equation}
where $k:\R^n\to \R$ is the heat kernel and $q=p/(p-1)$. Since $q\in (0,\infty)$, the $L_q$-norm of the heat kernel $k_{\tau/2}$ converges to $0$ as $\tau\to \infty$. Moreover, $\|e^{\frac{\tau}{2}\Delta}h\|_Y$ is dominated point-wise by the Hardy--Littlewood maximal function $M\|h\|_Y\in L_p(\R^n)$. Hence~\eqref{eq:heat limit at infty} follows from~\eqref{eq:heat kernel pointwise}   by an application of the dominated convergence theorem.
\end{proof}

The following theorem is the main result of the present section: it establishes an estimate that will be used several times in the ensuing discussion.

\begin{theorem}\label{thm:A representation} Fix $n\in \N$ and $p\in (1,\infty)$. Suppose that for $t\in (0,\infty)$ we are given a bounded operator $A(t):Y\to Y$ such that the mapping $A:(0,\infty)\to \bddlin(Y,Y)$ is bounded and measurable. Also, suppose that  $T:L_p(\R^n,Y)\to L_p(\R^n,Y)$ is  a linear operator that has the following dual representation. For every sufficiently nice $f$ in a dense subspace of $L_p(\R^n,Y)$, and $g^*$ in a dense subspace of $L_{q}(\R^n,Y^*)$, where $q=p/(p-1)$, we have
\begin{equation}\label{eq:representation A}
\int_{\R^n} g^*(x)\left(Tf(x)\right)\dd x=\sum_{j=1}^n\int_0^\infty\int_{\R^n} \bigg(\frac{\partial}{\partial x_j} e^{\frac{t}{2}\Delta}g^*(x)\bigg)\bigg(A(t)\frac{\partial}{\partial x_j}  e^{\frac{t}{2}\Delta}f(x)\bigg)\dd x \dd t.
\end{equation}
Then, recalling the notation~\eqref{eq:R_p(A) notation},
\begin{equation}
\|T\|_{L_p(\R^n,Y)\to L_p(\R^n,Y)}\le \beta_p^+(Y)\beta_p^-(Y)\mathscr{R}_p
(A)\label{eq:A rep bound}.
\end{equation}

Moreover, in the special case $A:(0,\infty)\to \C$  we have
\begin{equation}\label{eq:T bound real valued}
\|T\|_{L_p(\R^n,Y)\to L_p(\R^n,Y)}\le \beta_p(Y)\left\|A\right\|_{L_\infty(0,\infty)}.
\end{equation}
\end{theorem}

\begin{proof}
By duality and the identity~\eqref{eq:representation A}, the desired estimate~\eqref{eq:A rep bound} would follow if we show that for every sufficiently nice $f\in L_p(\R^n,Y)$ and $g^*\in L_q(\R^n,Y^*)$,
\begin{multline}\label{eq:tau limit}
\limsup_{\tau\to \infty} \sum_{j=1}^n\int_0^\tau\int_{\R^n}
\left(\frac{\partial}{\partial x_j}
e^{\frac{s}{2}\Delta}g^*(x)\right)\left(A(s)\frac{\partial}{\partial x_j} e^{\frac{s}{2}\Delta}f(x)\right)\dd x
\dd s\\\le
\beta_p^+(Y)\beta_p^-(Y)\mathscr{R}_p(A)\|f\|_{L_p(\R^n,Y)}\|g^*\|_{L_q(\R^n,Y^*)}.
\end{multline}
Let us consider $f,g$ smooth and compactly supported, and taking values in finite-dimensional subspaces of $Y$ and $Y^*$, respectively.

For every $(x,s)\in \R^n\times (0,\infty)$ denote for the sake of simplicity
$$
\gamma_j^*(x,s)\eqdef \frac{\partial}{\partial x_j}
e^{\frac{s}{2}\Delta}g^*(x)\qquad\mathrm{and}\qquad \phi_j(x,s)\eqdef A(s)\frac{\partial}{\partial x_j} e^{\frac{s}{2}\Delta}f(x).
$$
Let $\{B(t)=(B_1(t),\ldots,B_n(t))\}_{t\in [0,\infty)}$ be a Brownian motion in $\R^n$, starting at $0$. It follows from the identity~\eqref{eq:markov property} that for every $\tau\in (0,\infty)$,
\begin{multline}
\sum_{j=1}^n\int_0^\tau\int_{\R^n}\left(\frac{\partial}{\partial x_j}
e^{\frac{s}{2}\Delta}g^*(x)\right)\left(A(s)\frac{\partial}{\partial x_j} e^{\frac{s}{2}\Delta}f(x)\right)\dd x
\dd s\\=\int_{\R^n}\E\bigg[\sum_{j=1}^n\int_0^\tau\gamma_j^*(B(t)+x,\tau-t)\left(\phi_j(B(t)+x,\tau-t)\right)dt\bigg]\dd x =\int_{\R^n}\E\left[G^*_\tau(x)(F_\tau(x))\right]\dd x\label{eq:ito isometry},
\end{multline}
where we introduce the notations
\begin{equation}\label{eq:def F tau}
F_\tau(x)\eqdef \sum_{j=1}^n\int_0^\tau A(\tau-t)\frac{\partial}{\partial x_j}e^{\frac{\tau-t}{2}\Delta}f(B(t)+x) \dd B_j(t),
\end{equation}
and
\begin{equation}\label{eq:def G tau}
G^*_\tau(x)\eqdef \sum_{j=1}^n\int_0^\tau\frac{\partial}{\partial x_j}e^{\frac{\tau-t}{2}\Delta}g^*(B(t)+x)\dd B_j(t).
\end{equation}
\eqref{eq:ito isometry} is a well-known identity (see~\cite[Proposition 2.17]{KS}) for scalar-valued functions, and it follows from this for the vector-valued functions with a finite-dimensional range that we consider here.

By H\"older's inequality it follows from~\eqref{eq:ito isometry} that
\begin{multline}\label{eq:pq holder FG}
\sum_{j=1}^n\int_0^\tau\int_{\R^n}\left(\frac{\partial}{\partial x_j}
e^{\frac{s}{2}\Delta}g^*(x)\right)\left(A(s)\frac{\partial}{\partial x_j} e^{\frac{s}{2}\Delta}f(x)\right)\dd x
\dd s\\\le
\left(\int_{\R^n}\E\left[\|F_\tau(x)\|_Y^p\right]\dd x\right)^{\frac{1}{p}}\left(\int_{\R^n}\E\left[\|G_\tau^*(x)\|_{Y^*}^q\right]\dd x\right)^{\frac{1}{q}}.
\end{multline}
Recalling~\eqref{eq:def G tau}, by Lemma~\ref{lem:use ito} we have
\begin{equation}\label{eq:limsup Gtau}
\limsup_{\tau\to \infty} \left(\int_{\R^n}\E\left[\|G_\tau^*(x)\|_{Y^*}^q\right]\dd x\right)^{\frac{1}{q}}\le \|g^*\|_{L_q(\R^n,Y^*)}.
\end{equation}
Recalling~\eqref{eq:def F tau}, it follows from~\eqref{eq:R-boundedness stochastic integral-coupled} and~\eqref{eq:contraction stochastic integral} that if we set $K=\beta_p^+(Y)\beta_p^-(Y)\mathscr{R}_p(A)$ if $Y\neq \R$, and $K=\beta_p(Y)\|A\|_{L_\infty(0,\infty)}$ if $A$ is scalar-valued, then
\begin{equation*}
\int_{\R^n}\E\left[\|F_\tau(x)\|_Y^p\right]\dd x\le K^p\E\bigg[\Big\|\sum_{j=1}^n\int_0^\tau \frac{\partial}{\partial x_j}e^{\frac{\tau-t}{2}\Delta}f(B(t)+x) \dd B_j(t)\Big\|_Y^p\bigg].
\end{equation*}
 Another application of Lemma~\ref{lem:use ito}  now implies that
\begin{equation}\label{eq:limsup Ftau}
\limsup_{\tau\to \infty} \left(\int_{\R^n}\E\left[\|F_\tau(x)\|_{Y}^p\right]\dd x\right)^{\frac{1}{p}}\le \|f\|_{L_q(\R^n,Y)}.
\end{equation}
The desired estimate~\eqref{eq:tau limit} is a consequence of~\eqref{eq:pq holder FG}, \eqref{eq:limsup Gtau} and~\eqref{eq:limsup Ftau}.
\end{proof}

\subsection{Vector-valued multipliers} Let $(Y,\|\cdot\|_Y)$ be a Banach space. The Fourier transform of $f\in L_1(\R^n,Y)$ will be denoted below by $\F f:\R^n\to Y$, where we use the normalization
$$
\F f(x)\eqdef \frac{1}{(2\pi)^{n/2}}\int_{\R^n} e^{- i\langle x,y\rangle}f(x) \dd y.
$$
A possible formulation of Parseval's identity in this vector-valued setting is to say that for functions $f:\R^n\to Y$ and $g^*:\R^n\to Y^*$ that are either smooth and compactly supported, or Fourier transforms of such functions, we have
\begin{equation}\label{eq:parseval vector valued}
\int_{\R^n} \F g^*(x)(\F f(x))\dd x=\int_{\R^n} g^*(x)(f(-x))\dd x.
\end{equation}
If $(X,\|\cdot\|_X)$ is an additional Banach space and $\mathfrak{m}:\R^n\to\bddlin(X,Y)$ is measurable then the multiplier associated to $\mathfrak{m}$ is defined as usual by considering for every smooth and compactly supported $f: \R^n\to X$ (or the Fourier transform of such a function) the function $T_{\mathfrak{m}}f:\R^n\to Y$ given by
$$
T_{\mathfrak{m}}f\eqdef (\F^{-1}\mathfrak{m})*f=\F^{-1}\left(x\mapsto \mathfrak{m}(x)\F f(x)\right).
$$
If $\mathfrak{m}$ is smooth and locally bounded at least away from the coordinate hyperplanes, and the Fourier transform of $f: \R^n\to X$ is smooth and compactly supported away from these hyperplanes, then also $T_{\mathfrak m}f$ has a smooth and compactly supported Fourier transform. Hence, for such $f$ and smooth and compactly supported $g^*:\R^n \to Y^*$ (or the Fourier transform of such a function), Parseval's identity applies and gives
\begin{equation}\label{eq:parseval multiplier}
\int_{\R^n} g^*(x)(T_{\mathfrak{m}}f(x))\dd x= \int_{\R^n}\F g^*(x)\left(\mathfrak{m}(-x)\F f(-x)\right)\dd x.
\end{equation}
Also, under our choice of normalization of the Fourier transform, $-\Delta=T_\mathfrak{m}$ for $\mathfrak{m}(x)=\|x\|_2^2$.

Theorem~\ref{thm:A representation} can be used to bound the following multipliers, which arise as Laplace transforms of $-\Delta$. (Another approach to such multipliers appears in the recent survey \cite[Section 2.2.1]{DGM}; while it is presented there for scalar-valued functions, it is based on principles that are valid in any UMD space.) Suppose that $A:(0,\infty)\to \bddlin(Y,Y)$ is measurable and define $\alpha:\R\to \bddlin(Y,Y)$ by
$$
\forall\, y\in Y,\qquad \alpha(s)\eqdef s\int_0^\infty e^{-st}A(t)y\dd t.
$$
Then Theorem~\ref{thm:A representation} applies to $\alpha(-\Delta)$, i.e., to the operator $T_\mathfrak{m}$ where $\mathfrak{m}:\R^n\to \bddlin(Y,Y)$ is given by
\begin{equation}\label{eq:laplacian function multiplier}
\forall (x,y)\in \R^n\times Y,\qquad \mathfrak{m}(x)y=\alpha(\|x\|_2^2)y=\|x\|_2^2\int_0^\infty e^{-t\|x\|_2^2}A(t)y\dd t.
\end{equation}
Indeed, by Parseval's identity~\eqref{eq:parseval multiplier}, the representation~\eqref{eq:representation A} is a direct consequence of~\eqref{eq:laplacian function multiplier}. We therefore have the following dimension independent bound, which holds true for every $p\in (1,\infty)$.
$$
\left\|\Delta\int_0^\infty e^{t\Delta}A(t)\dd t\right\|_{L_p(\R^n,Y)\to L_p(\R^n,Y)}\le \beta_p(Y)^2 \mathscr{R}_p(A),
$$
where we recall the notation~\eqref{eq:R_p(A) notation}. Also, if $A$ takes values in $\C$ then
\begin{equation}\label{eq:laplace of laplace}
\left\|\Delta\int_0^\infty e^{t\Delta}A(t)\dd t\right\|_{L_p(\R^n,Y)\to L_p(\R^n,Y)}\le \beta_p(Y)\|A\|_{L_\infty(0,\infty)}.
\end{equation}

Later we shall use~\eqref{eq:laplace of laplace} as a source of dimension-independent bounds for multipliers that correspond to imaginary powers of the Laplacian. Specifically, for every $s\in (0,\infty)$ and $u\in \R$,
$$
s^{iu} =s^{-(1-iu)}s=\frac{s}{\Gamma(1-iu)}\int_0^\infty t^{-iu}e^{-st}\dd t.
$$
It therefore follows from~\eqref{eq:laplace of laplace} that
\begin{equation}\label{eq:imaginary power of laplacian}
\left\|(-\Delta)^{iu}\right\|_{L_p(\R^n,Y)\to L_p(\R^n,Y)}\le \frac{\beta_p(Y)}{|\Gamma(1-iu)|}\asymp \beta_p(Y)\frac{e^{|u|\arctan|u|}}{\sqrt{1+|u|}}
\asymp \beta_p(Y) \frac{e^{\frac{\pi |u|}{2}}}{\sqrt{1+|u|}},
\end{equation}
where the penultimate step in~\eqref{eq:imaginary power of laplacian} is a consequence of Stirling's formula. For ease of later reference, we record the bound that we have just proved as Corollary~\ref{cor:imaginary power laplacian} below. The reverse implication, i.e., that the boundedness  $(-\Delta)^{iu}$ on $L_p(\R^n,Y)$ implies that $Y$ is UMD, is also true; in fact it was pointed out in~\cite{Hytonen:aspects}  that~\cite{Guerre} implicitly contains the estimate $\beta_p(Y)\leq\liminf_{u\to 0}\|(-\Delta)^{iu}\|_{L_p(\R^n,Y)\to L_p(\R^n,Y)}$.

\begin{corollary}\label{cor:imaginary power laplacian}
Suppose that $p\in (1,\infty)$ and $(Y,\|\cdot\|_Y)$ is a UMD Banach space. Then for every $u\in \R$ and $n\in \N$ we have
$$
\left\|(-\Delta)^{iu}\right\|_{L_p(\R^n,Y)\to L_p(\R^n,Y)}\lesssim \beta_p(Y) \frac{e^{\frac{\pi |u|}{2}}}{\sqrt{1+|u|}}.
$$
\end{corollary}

Our next corollary of Theorem~\ref{thm:A representation} is a dimension-independent bound for a multiplier that will be used in the ensuing proof of Theorem~\ref{thm:Lp UAAP intro}. We shall use below the following integral representation.
\begin{equation}\label{eq:beta identity}
\forall (\theta,\alpha)\in (0,1)\times [0,\infty),\qquad \frac{1}{(1+\alpha)^\theta}=\frac{\sin(\pi\theta)}{\pi}\int_0^1 \frac{\dd s}{s^{1-\theta}(1-s)^\theta(1+\alpha s)}.
\end{equation}
To verify the validity of~\eqref{eq:beta identity}, simply apply the change of variable $s=t/(1+\alpha-\alpha t)$.

\begin{corollary}\label{cor:a power mult}
Fix $a\in (0,2]$ and $n\in \N$. Define $\mathfrak{m}_a:\R^n\to \R$ by setting
\begin{equation}\label{eq:def ma}
\forall, x=(x_1,\ldots,x_n)\in \R^n\setminus \{0\},\qquad \mathfrak{m}_a(x)\eqdef \frac{|x_1|^a}{\|x\|_2^a}.
\end{equation}
Suppose that $(Y,\|\cdot\|_Y)$ is a UMD Banach space and that $p\in (1,\infty)$. Then
\begin{equation}\label{eq:m a multiplier bound}
\|T_{\mathfrak{m}_a}\|_{L_p(\R^n,Y)\to L_p(\R^n,Y)}\le \beta_p^+(Y)^2\beta_p^-(Y)\le \beta_p(Y)^3.
\end{equation}
\end{corollary}

\begin{remark}
{\em When $Y=\C$ in Corollary~\ref{cor:a power mult}, Ba\~nuelos and Bogdan~\cite{BB07} obtained the bound
\begin{equation}\label{eq:BB}\|T_{\mathfrak{m}_a}\|_{L_p(\R^n,\R)\to L_p(\R^n,\R)}\le \max\left\{p,\frac{p}{p-1}\right\}-1.
 \end{equation}
Note that $\beta_p(\C)=\max\{p,p/(p-1)\}-1$, by a theorem of Burkholder~\cite{Bur84}. We are unable to recover the better estimate~\eqref{eq:BB} for the scalar-valued case of  Corollary~\ref{cor:a power mult} using our method.}
\end{remark}

\begin{proof}[Proof of Corollary~\ref{cor:a power mult}] Write each $x\in \R^n$ as $x=(x_1,x')$, where $x'=(x_2,\ldots,x_n)\in \R^{n-1}$. For every $s\in (0,1]$ define $\gamma_s:\R^n\to \R$ by
\begin{equation}\label{eq:define auxilliary gamma n-1}
\forall\, x\in \R^n,\qquad \gamma_s(x)\eqdef \frac{x_1^2}{x_1^2+s\|x'\|_2^2}=\int_0^\infty x_1^2e^{-(x_1^2+s\|x'\|_2^2)t}\dd t.
\end{equation}
Then $\gamma_1=\mathfrak{m}_2$, and if $a\in (0,2)$ then by~\eqref{eq:beta identity} with $\theta=a/2\in (0,1)$ and $\alpha=\|x'\|_2^2/x_1^2$ we have
\begin{equation*}\label{eq:ma identity}
\mathfrak{m}_a=\int_0^1 \gamma_sd\mu_a(s),
\end{equation*}
where $\mu_a$ is the probability measure on $(0,1)$ whose density is proportional to $s^{\frac{a}{2}-1}(1-s)^{-\frac{a}{2}}$. Therefore, in order to prove the desired estimate~\eqref{eq:m a multiplier bound} it suffices to show that for every $s\in (0,1)$,
$$
\|T_{\gamma_s}\|_{L_p(\R^n,Y)\to L_p(\R^n,Y)}\le \beta_p^+(Y)^2\beta_p^-(Y).
$$

Consider the UMD Banach space $Z\eqdef L_p(\R^{n-1},Y)$. By the identification $L_p(\R^n,Y)\cong L_p(\R,Z)$, the multiplier $T_{\gamma_s}$ can be thought of as an operator from $L_p(\R,Z)$ to $L_p(\R,Z)$; this is how Theorem~\ref{thm:A representation} will be applied next, i.e., with $Y$ replaced by $Z$, while noting that $\beta_p^\pm(Z)=\beta_p^\pm(Y)$.

We consider a test function $f:\R^n\to Y$ that is finite linear combination of functions of the form $x\mapsto f_1(x_1)f_2(x_2)\cdots f_n(x_n)y$, where $y\in Y$ and each $f_i$ has a smooth Fourier transform, compactly supported away from $0$, and a similar function $g^*:\R^n\to Y^*$. Note that such functions are dense in $L_p(\R,Z)\cong L_p(\R^n,Y)$ and in $L_q(\R,Z^*)=L_q(\R,L_q(\R^{n-1},Y^*))\cong L_q(\R^n,Y^*)$, respectively, where $q=p/(p-1)$.
 Then by the Parseval identity~\eqref{eq:parseval multiplier} we have
$$
\int_\R g^*(x_1)(T_{\gamma_s}f(x_1))\dd x_1=\int_{\R^n} g^*(x)(T_{\gamma_s}f(x))\dd x=\int_{\R^n} \F g^*(x)\left(\frac{x_1^2}{x_1^2+s\|x'\|_2^2}\F f(-x)\right)\dd x.
$$
Consequently, by the second equality in~\eqref{eq:define auxilliary gamma n-1},
\begin{multline*}
\int_\R g^*(x_1)(T_{\gamma_s}f(x_1))\dd x_1=\int_0^\infty \int_{\R^n} x_1 e^{-\frac{t }{2}x_1^2}\F g^*(x) \left(e^{-ts \|x'\|_2^2}x_1 e^{-\frac{t}{2} x_1^2}\F f(-x)\right)\dd x\dd t\\
=\int_0^\infty \int_{\R^n} \F\left(\frac{\partial}{\partial x_1} e^{\frac{t}{2}\left(\frac{\partial}{\partial x_1}\right)^2}g^*\right)(x)\left(\F\left(-e^{ts\Delta'}\frac{\partial}{\partial x_1} e^{\frac{t}{2}\left(\frac{\partial}{\partial x_1}\right)^2}f\right)\right)(-x)\dd x\dd t,
\end{multline*}
where $\Delta'$ denotes the Laplacian on $\R^{n-1}$, i.e., with respect to the variable $x'$. By the vector-valued Parseval identity~\eqref{eq:parseval vector valued}, we therefore have
$$
\int_\R g^*(x_1)(T_{\gamma_s}f(x_1))\dd x_1=\int_0^\infty \int_\R \left(\frac{\partial}{\partial x_1}  e^{\frac{u}{2}\left(\frac{\partial}{\partial x_1}\right)^2}g^*(x_1)\right)\left(A(u)\frac{\partial}{\partial x_1}
 e^{\frac{u}{2}\left(\frac{\partial}{\partial x_1}\right)^2}f\right)(x_1)\dd x_1\dd u,
$$
where $A(u)\eqdef -e^{us\Delta'}: L_p(\R,Z)\to L_p(\R,Z)$. Recalling Corollary~\ref{cor:heat}, it therefore follows from Theorem~\ref{thm:A representation} that
\begin{multline*}
\|T_{\gamma_s}\|_{L_p(\R^n,Y)\to L_p(\R^n,Y)}=\|T_{\gamma_s}\|_{L_p(\R,Z)\to L_p(\R,Z)}\\\le
\beta_p^+(Z)\beta_p^-(Z)\mathscr{R}_p\left(\left\{e^{t\Delta'}:L_p(\R^{n-1},Y)\to L_p(\R^{n-1},Y)\right\}_{t\in (0,\infty)}\right)\le \beta_p^+(Y)^2\beta_p^-(Y).\tag*{\qedhere}
\end{multline*}
\end{proof}

\subsection{Littlewood--Paley decomposition}
We need to introduce notation for the usual multi-scale bump functions that occur in Littlewood--Paley decompositions.
Let $\phi:\R\to [0,\infty)$ be smooth and supported on  $[-2,-1/2]\cup[1/2,2]$; for concreteness we can take
\begin{equation}\label{eq:def phi LP}
\forall\, x\in \R,\qquad \phi(x)\eqdef \left\{\begin{array}{ll}
e^{-\frac{1}{(|x|-1/2)(2-|x|)}}& \mathrm{if}\ |x|\in (1/2,2),\\
0& \mathrm{if}\ |x|\in [0,1/2]\cup [2,\infty).
\end{array}\right.
\end{equation}
For $k\in \Z$ define $\psi_k:\R\to \R$ by
\begin{equation}\label{eq:def psik}
\forall\, x\in \R,\qquad \psi_k(x)\eqdef \frac{\phi(2^kx)}{\sum_{j\in \Z} \phi(2^j x)}.
\end{equation}
We also define $\omega_k:\R\to \R$ by
\begin{equation}\label{eq:def omegak}
\forall\, x\in \R,\qquad \omega_k(x)\eqdef \psi_{k-1}(x)+\psi_k(x)+\psi_{k+1}(x).
\end{equation}
Thus $\omega_k\psi_k=\psi_k$, and therefore the corresponding multipliers satisfy the identity $T_{\omega_k}T_{\psi_k}=T_{\psi_k}$.

For every $k\in \Z$ define $\t_k:\R\to [0,\infty)$ by
\begin{equation}\label{eq:def thetha k}
\forall\, x\in \R,\qquad \t_k(x)\eqdef \frac{\sin^4(2^k x)}{(2^k x)^2}.
\end{equation}
Like  $\psi_k$, the function $\t_k$ is roughly localized around $|x|\asymp 2^{-k}$, but unlike $\psi_k$, it has long tails that are supported over all of $\R$. The importance of the special ``pseudo-bump functions" $\{\t_k\}_{k\in \Z}$ stems from the fact that they can be directly related to averages of dyadic martingales, which leads to the following form of the Littlewood--Paley inequality with a good constant.

\begin{proposition}\label{prop:littlewood paley}
Suppose that $p\in (1,\infty)$ and let $(Y,\|\cdot\|_Y)$ be a UMD Banach space. Then
\begin{equation*}
\forall\, f\in L_p(\R,Y),\qquad   \bigg(\E_{\e\in \{-1,1\}^{\Z}}\Big\|\sum_{j\in\Z}\e_j T_{\t_j}f\Big\|_{L_p(\R,Y)}^p\bigg)^{\frac{1}{p}}
  \leq\beta_p^+(Y)\|f\|_{L_p(\R,Y)}\le \beta_p(Y)\|f\|_{L_p(\R,Y)}.
\end{equation*}
\end{proposition}
Proposition~\ref{prop:littlewood paley} is due implicitly to Bourgain~\cite{Bourgain:86}, where it is proved as an intermediate step towards a more usual form of the Littlewood--Paley inequality involving the localized bump functions $\{\psi_k\}_{k\in \Z}$ in place of $\{\t_k\}_{k\in \Z}$, but with a more complicated dependence on the UMD constant. The above formulation of Proposition~\ref{prop:littlewood paley} appears explicitly as the special case $h=k=\1_{[0,1/2]}-\1_{[1/2,1]}$ of Proposition~5.10 in~\cite{HNVW} (where we are using here the notation of~\cite[Proposition~5.10]{HNVW}).

 In subsequent arguments we shall use Proposition~\ref{prop:littlewood paley} in addition to some  auxiliary estimates concerning the bump functions $\{\psi_k\}_{k\in Z}$, which are valid for arbitrary Banach space targets. These estimates rely on the fact that if $\mathfrak{m}\in L_1(\R)$  then by Young's inequality we have
\begin{equation}\label{eq:easy multiplier}
 \forall\, f\in L_p(\R,Y),\qquad  \|T_{\mathfrak{m} } f\|_{L_p(\R,Y)}=\left\|\left(\mathcal{F}^{-1}\mathfrak{m}\right)*f\right\|_{L_p(\R,Y)}\leq\left\|\mathcal{F}^{-1}\mathfrak{m}\right\|_{L_1(\R)}\|f\|_{L_p(\R,Y)}.
\end{equation}

\begin{lemma}\label{lem:psi vs theta} Suppose that $p\in [1,\infty]$ and that $(Y,\|\cdot\|_Y)$ is a Banach space. Then for every $k\in \Z$,
\begin{equation}\label{psi theta comparison}
\forall\, f\in L_p(\R,Y),\qquad   \Norm{T_{\psi_k}f}{L_p(\R,Y)}\lesssim\Norm{T_{\t_k}f}{L_p(\R,Y)}.
\end{equation}
\end{lemma}

\begin{proof} Since for every $x\in \R$ and $k\in \Z$ we have $\t_k(x)=\t_0(2^kx)$ and $\psi_k(x)=\psi_0(2^k x)$, we also have $T_{\t_k}f(x)=T_{\t_0}(y\mapsto f(2^ky))(2^{-k}x)$ and $T_{\psi_k}f(x)=T_{\psi_0}(y\mapsto f(2^ky))(2^{-k}x)$. Consequently, it suffices to prove~\eqref{psi theta comparison} when $k=0$.  Since $\t_0$ is nonzero on the support of $\psi_0$, we can write $T_{\psi_0}f=T_{\psi_0/\t_0}T_{\theta_0}f$. The function $\psi_0/\t_0$ is smooth and compactly supported, therefore its inverse Fourier transform $g=\mathcal{F}^{-1}(\psi_0/\t_0)$ belongs to the Schwartz class of test functions $\mathscr{S}(\R)$, and in particular $g\in L_1(\R)$. Thus~\eqref{psi theta comparison} follows from~\eqref{eq:easy multiplier} with $\mathfrak{m}=\psi_0/\t_0$ and $T_{\t_0}f$ in place of $f$.
\end{proof}

 In order to facilitate the next two applications of \eqref{eq:easy multiplier}, we record the following simple observation. If $\mathfrak{m}\in L_1(\R)$ is smooth then for every $a\in (0,\infty)$ we have
\begin{multline}\label{eq:L1 vs FT}
  \left\|\mathcal{F}^{-1}\mathfrak{m}\right\|_{L_1(\R)}\le  \bigg(\int_\R \frac{\dd x}{1+(ax)^2}\bigg)\sup_{x\in \R} \left|(1+a^2x^2)\left(\mathcal{F}^{-1}\mathfrak{m}\right)(x)\right|\\=\frac{\pi}{a}\left\|\mathcal{F}^{-1}(\mathfrak{m}-a^2\mathfrak{m}'')\right\|_{L_\infty(\R)}
  \le\frac{\pi}{a} \left\|\mathfrak{m}-a^2\mathfrak{m}''\right\|_{L_1(\R)}\le \frac{\pi}{a} \left\|\mathfrak{m}\right\|_{L_1(\R)}+\pi a \left\|\mathfrak{m}''\right\|_{L_1(\R)}.
\end{multline}
Choosing $a=\sqrt{\|\mathfrak{m}\|_{L_1(\R)}/\|\mathfrak{m}''\|_{L_1(\R)}}$ in~\eqref{eq:L1 vs FT} and substituting the resulting estimate into~\eqref{eq:easy multiplier} yields
\begin{equation}\label{eq:easy multiplier2}
 \forall\, f\in L_p(\R,Y),\qquad  \|T_{\mathfrak{m} } f\|_{L_p(\R,Y)}\leq 2\pi\sqrt{\|\mathfrak{m}\|_{L_1(\R)}\|\mathfrak{m}''\|_{L_1(\R)}}\cdot \|f\|_{L_p(\R,Y)}.
\end{equation}

\begin{lemma}\label{lem:localization}
Fix $k\in \Z$ and  $p\in [1,\infty]$. Let $(Y,\|\cdot\|_Y)$ be a Banach space. Then for every $f\in L_p(\R,Y)$ and $y\in \R$ we have
\begin{equation}\label{eq:min y}
\left(\int_\R\left\|T_{\psi_k}f(x+y)-T_{\psi_k(x)}f(x)\right\|_Y^p\dd x\right)^{\frac{1}{p}}\lesssim \min\left\{1,\frac{|y|}{2^k}\right\}\left\|T_{\psi_k}f\right\|_{L_p(\R,Y)}.
\end{equation}
\end{lemma}

\begin{proof}
Write $g(x)\eqdef T_{\psi_k}f(x+y)-T_{\psi_k}f(x)$. The fact that the norm of $g$ in $L_p(\R,Y)$ is at most twice the norm of $T_{\psi_k}f$ in $L_p(\R,Y)$ follows from the triangle inequality in $L_p(\R,Y)$. For the second estimate in~\eqref{eq:min y}, note that $g=T_{\rho_y\omega_k}T_{\psi_k}f$, where $\rho_y(x)=e^{-ixy}-1$. Hence, it remains to show that for every $h\in L_p(\R,Y)$ and $y\in [-2^{k},2^k]$ we have $\|T_{\rho_y \omega_k}h\|_{L_p(\R,Y)}\lesssim |y|2^{-k}\|h\|_{L_p(\R,Y)}$. By~\eqref{eq:easy multiplier2}, this will follow if we prove that $\|\rho_y \omega_k\|_{L_1(\R)}\cdot \|(\rho_y\omega_k)''\|_{L_1(\R)}\lesssim (2^{-k}y)^2$. Since $\rho_y(x)\omega_k(x)=(\rho_{2^{-k}y}\omega_0)(2^ky)$ and for every smooth $\mathfrak{m}\in L_1(\R)$ the product $\|x\mapsto \mathfrak{m}(\lambda x)\|_{L_1(\R)}\cdot \|(x\mapsto \mathfrak{m}(\lambda x))''\|_{L_1(\R)}$ is independent of $\lambda\in (0,\infty)$, it suffices to show that \begin{equation}\label{eq:rho z}
\forall\, z\in [-1,1],\qquad \|\rho_z\omega_0\|_{L_1(\R)}\lesssim |z|\qquad \mathrm{and}\qquad \|(\rho_z\omega_0)''\|_{L_1(\R)}\lesssim |z|.
\end{equation}

The point-wise estimate $|\rho_z(x)|\le |z x|$ combined with the fact that the function $x\mapsto x\omega_0(x)$ is in $L_1(\R)$ implies the first assertion in~\eqref{eq:rho z}.  For the second assertion in~\eqref{eq:rho z}, compute directly that
\begin{equation*}
  (\rho_z\omega_0)''(x)=-z^2e^{-ixz}\omega_0(x)-i2ze^{-ixz}\omega_0'(x)+(e^{-ixz}-1)\omega_0''(x).
\end{equation*}
We therefore have the following point-wise estimate, which holds true whenever $|z|\le 1$.
\begin{equation*}
  |(\rho_z\omega_0)''(x)|\leq |z|^2 |\omega_0(x)|+2 |z\omega_0'(x)|+|zx\omega_0''(x)|\le |z|\big( |\omega_0(x)|+2 |\omega_0'(x)|+ |x\omega_0''(x)|\big).
\end{equation*}
The second assertion in~\eqref{eq:rho z} is now a consequence of the fact that the three functions $\omega_0,\omega_0'$ and $x\mapsto x\omega_0''(x)$ all belong to $L_1(\R)$. This concludes the proof of Lemma~\ref{lem:localization}\qedhere
\end{proof}

\begin{lemma}\label{lem:localization2}
Fix $k\in \Z$,  $p\in [1,\infty]$ and $\alpha\in (0,\infty)$. Let $(Y,\|\cdot\|_Y)$ be a Banach space. Then for every smooth and compactly supported $f:\R\to Y$ we have
\begin{equation}\label{eq:1 plus alpha}
\left\|T_{\psi_k}f\right\|_{L_p(\R,Y)}\lesssim (1+\alpha) 2^{2\alpha(k+2)} \left\|T_{\psi_k}(-\Delta)^{\alpha}f\right\|_{L_p(\R,Y)}.
\end{equation}
\end{lemma}

\begin{proof} Recalling the definition of $\omega_k:\R\to [0,\infty)$ that is given in~\eqref{eq:def omegak}, since $T_{\omega_k}T_{\psi_k}=T_{\psi_k}$
we have $T_{\psi_k}f=(-\Delta)^{-\alpha}T_{\omega_k} T_{\psi_k}(-\Delta)^{\alpha}f=T_{\xi_k}T_{\psi_k}(-\Delta)^{\alpha}f$, where
\begin{equation*}
\forall\, x\in \R\setminus \{0\},\qquad  \xi_k(x)\eqdef \frac{\omega_k(x)}{|x|^{2\alpha}}
  =2^{2k\alpha}\frac{\omega_0(2^k x)}{|2^k x|^{2\alpha}}
  =2^{2k\alpha}\xi_0(2^k x).
\end{equation*}
It therefore suffices to show that $\xi_0$ is the Fourier transform of a function in $L_1(\R)$ of norm at most a constant multiple of $(1+\alpha) 2^{4\alpha}$. To this end, we will again apply the estimate~\eqref{eq:easy multiplier2}.
First, since $|x|^{-1}\leq 4$ on the support of $\omega_0$ and $\omega_0\in L_1(\R)$, we have $\|\xi_0\|_{L_1(\R)}\lesssim 2^{4\alpha}$.
Also,
\begin{equation*}
\forall\, x\in \R\setminus\{0\},\qquad   \xi_0''(x)=\frac{2\alpha(2\alpha+1)}{|x|^{2\alpha}}\cdot \frac{\omega_0(x)}{x^2}-\frac{4\alpha}{|x|^{2\alpha}}\cdot \frac{\omega_0'(x)}{|x|}+\frac{\omega_0''(x)}{|x|^{2\alpha}}.
\end{equation*}
Using again that $|x|^{-1}\leq 4$ on the support of $\omega_0$, combined with the fact that  the three functions $\omega_0''$, $x\mapsto \omega_0(x)/x^2$ and $x\mapsto \omega_0'(x)/x$ are all in $L_1(\R)$, it follows that $\|\xi_0''\|_{L_1(\R)}\lesssim (2\alpha+1)^22^{4\alpha}$. So,
\begin{equation*}
\|T_{\xi_0}\|_{L_p(\R,Y)\to L_p(\R,Y)}\stackrel{\eqref{eq:easy multiplier2}}{\lesssim} \sqrt{\|\xi_0\|_{L_1(\R)}\|\xi_0''\|_{L_1(\R)}}\le \sqrt{4^{2\alpha}\cdot (2\alpha+1)^24^{2\alpha}}\lesssim (1+\alpha)4^{2\alpha}. \qedhere
\end{equation*}
\end{proof}

\subsection{Type and Cotype}\label{sec:cotype} For $p\in [1,2)$ and $q\in [2,\infty)$, the type $p$ constant and the cotype $q$ constant of a Banach space $(Y,\|\cdot\|_Y)$, denoted $T_p(X)$ and $C_q(Y)$, respectively, are defined to be the infimum over those $T,C\in [1,\infty]$ such that for every $n\in\N$ and every $x_1,\ldots,x_n\in Y$,
\begin{equation}\label{eq:type cotype def}
\frac{1}{C}\Big(\sum_{j=1}^n \|x_j\|_Y^q\Big)^{\frac{1}{q}}\le \E\bigg[\Big\|\sum_{j=1}^n \e_j x_j\Big\|_Y\bigg]\le T\Big(\sum_{j=1}^n \|x_j\|_Y^p\Big)^{\frac{1}{p}},
\end{equation}
where the expectation is with respect to $\e=(\e_1,\ldots,\e_n)\in \{-1,1\}^n$ chosen uniformly at random. The smallest $T,C\in (0,\infty]$ for which~\eqref{eq:type cotype def} holds true are denoted $T_p(Y),C_q(Y)$, respectively.  Any UMD Banach space $(Y,\|\cdot\|_Y)$ admits an equivalent uniformly convex norm~\cite{Mau75,Ald79}, and hence it has finite cotype~\cite{FP74,Pis75-martingales}. The following lemma makes this qualitative statement quantitative in the case of cotype. A similar (and simpler) argument yields a quantitative bound in the case of type as well (see Remark~\ref{rem:np duality type} below), but in what follows only the case of cotype will be used.

\begin{lemma}[Cotype in terms of $\beta(Y)$]\label{lem:cotype bound}
There exists a universal constant $\kappa\in (1,\infty)$ such that for every $\beta\in [1,\infty)$, if $(Y,\|\cdot\|_Y)$ is a UMD Banach space with $\beta(Y)\le \beta$ then then $(Y,\|\cdot\|_Y)$ it has cotype $\kappa\beta$, and moreover $C_{\kappa \beta}(Y)\le \kappa$.
\end{lemma}

\begin{proof}
The proof below is a (somewhat tedious) combination of several results that appear in the literature. The key step is an examination of the proof of Pisier's quantitative version~\cite{Pis83} of the Maurey--Pisier theorem~\cite{MP76} for stable type. For $p\in (1,2)$ let $ST_p(Y^*)$ be the stable type $p$ constant of the dual space $Y^*$. Namely,  $ST_p(Y^*)$ is the infimum over those $S\in (0,\infty]$ such that for every $n\in \N$, every $x_1^*,\ldots,x_n^*\in Y^*$ satisfy
\begin{equation}\label{eq:STp}
\mathbb{E}\bigg[\Big\|\sum_{j=1}^n \theta_j x_j^*\Big\|_{Y^*}\bigg]\le S\Big(\sum_{j=1}^n \left\|x_j^*\right\|_{Y^*}^p\Big)^{\frac{1}{p}},
\end{equation}
where in~\eqref{eq:STp} the expectation is with respect to i.i.d. standard symmetric $p$-stable random variables $\{\theta_j\}_{j=1}^n$, i.e., the characteristic function of $\theta_1$ is
\begin{equation}\label{eq:p stable def}
\forall\, t\in \R,\qquad \E\left[e^{i t\theta_1}\right]=e^{-|t|^p}.
\end{equation}
It  follows from~\eqref{eq:p stable def} that $\E\left[|\theta _1|\right]\asymp 1/(p-1)$; see e.g.~\cite[Sec.~XVII]{Fel71}. By Jensen's inequality and Kahane's inequality, this implies that $T_p(Y^*)\le (p-1) ST_p(Y^*)$. Since $C_q(Y)\le T_p(Y^*)$, where $q=p/(p-1)$ (see e.g. \cite[Sec.~6]{Mau03}), we deduce that
\begin{equation}\label{eq:stp lower bound}
ST_p(Y^*)\gtrsim \frac{C_q(Y)}{p-1}\asymp qC_q(Y).
\end{equation}

Suppose from now on that $p\in (1,3/2]$ (in the argument below we only use that $p$ is bounded away from $2$ by a universal constant). Equivalently, $q\in [3,\infty)$.
It follows from the proof of the main theorem of~\cite{Pis83} that there exists a universal constant $c\in (0,1)$ such that if $m\in \N$ satisfies
\begin{equation}\label{eq:pisier m bound}
m^{\frac{1}{q}}\le \frac{c}{q}ST_p(Y^*) \le c(p-1)ST_p(Y^*).
\end{equation}
then $Y^*$ contains a $2$-isomorphic copy of $\ell_p^m$.  Unfortunately, while the dependence of $m$ on $ST_p(Y^*)$ that is stated in~\eqref{eq:pisier m bound} is also stated explicitly in~\cite{Pis83}, the dependence on $p$, which is crucial for us here, is not computed in~\cite{Pis83}. However, one can verify~\eqref{eq:pisier m bound} by examining the dependencies on $p$ of certain constants that appear in~\cite{Pis83}, and substituting these dependencies into the proof of~\cite{Pis83}. Specifically,  the constant $C_p$ of Proposition 1.3 of~\cite{Pis83} was computed in~\cite{MP84} to be
$$
C_p=\left(\frac{1}{\int_0^\infty \frac{\sin v}{v^p}dv}\right)^{\frac{1}{p}}=2\left(\frac{\Gamma\left(\frac{p+1}{2}\right)}{\sqrt{\pi} \Gamma\left(1-\frac{p}{2}\right)}\right)^{\frac{1}{p}}.
$$
Thus, recalling that $p\in (1,3/2)$, we see that $C_p$ is bounded above and below by positive universal constants. The parameter $\Phi$ of Lemma 1.4 of~\cite{Pis83} can be estimated via a direct computation (e.g., using the last line of page 975 of~\cite{FG11}) to give $\Phi\lesssim 1/(p-1)$. The proof of~\cite{Pis83} uses only two additional unspecified parameters, denoted $K$ and $\eta$, that appear in Lemma 1.5 of~\cite{Pis83}. In Proposition 2 of~\cite{JS82}  it is shown that one can take $K=2$ and $\eta=(2-p)/(8p(q+1)^q)$. A direct substitution of these estimates into the proof of Theorem 2.1 of~\cite{Pis83} now yields~\eqref{eq:pisier m bound}.

Note that
\begin{equation}\label{eq:beta of ellp}
\beta(\ell_p^m)\gtrsim \min \left\{q,\left(\log m\right)^{\frac{1}{p}}\right\}=\min\left\{q,\left(\log m\right)^{1-\frac{1}{q}}\right\}.
\end{equation}
While~\eqref{eq:beta of ellp} is folklore, we did not find it in the literature so we briefly sketch the relevant computation.  Let $k\in \N$ be the largest integer such that $2^k\le m$. Let $\mu$ be the uniform probability measure on the discrete hypercube $\{-1,1\}^k$, and think of $\ell_p^m$ as containing an isometric copy of $L_p(\mu)$. Define $M_0,\ldots,M_k:\{-1,1\}^k\to L_p(\mu)$ by setting $M_0\equiv 1$ and for $j\in \{1,\ldots,k\}$ defining
 $$\forall\, \e,\d\in \{-1,1\}^k,\qquad M_j(\e)(\d)\eqdef \prod_{\ell=1}^k(1+\e_\ell\d_\ell)=2^k\1_{\{(\e_1,\ldots,\e_j)=(\d_1,\ldots,\d_j)\}}.$$
Then $\{M_j\}_{j=0}^k$ is a martingale with respect to the natural coordinate filtration of $\{-1,1\}^k$.

Observe that
\begin{equation}\label{eq:Mk upper}
\bigg(\int_{\{-1,1\}^k} \left\|M_k(\e)\right\|_{L_p(\mu)}^2\dd \mu(\e)\bigg)^{\frac12}=2^{\frac{k(p-1)}{p}}=2^{\frac{k}{q}}.
\end{equation}
For every $\e,\d\in \{-1,1\}^k$ write $j(\e,\d)=0$ if $\e_1\neq \d_1$ and otherwise let $j(\e,\d)$ be the largest $j\in \{1,\ldots,k\}$ such that $\epsilon_i = \delta_i$ for all $i\in \{1,\ldots,j\}$.  With this notation,
\begin{equation}\label{eq:S lower}
S(\e,\d)\eqdef \Big(\sum_{j=1}^k \left(M_j(\e)(\d)-M_{j-1}(\e)(\d)\right)^2\Big)^{\frac12}\asymp 2^{j(\e,\d)}-1.
\end{equation}
Note that for every $j\in \{0,\ldots,k\}$ we have
\begin{equation}\label{eq:product measure estimate j}
\mu\times \mu\left(\left\{(\e,\d)\in \{-1,1\}^k\times \{-1,1\}^k:\ j(\e,\d)=j\right\}\right)\asymp \frac{1}{2^{j}}.
\end{equation}
Hence, using the triangle inequality in $L_p(\mu)$ and Khinchine's inequality, we have
\begin{multline*}
\bigg(\int_{\{-1,1\}^k}\int_{\{-1,1\}^k}\Big\|\sum_{j=1}^k \eta_j\left(M_j(\e)-M_{j-1}(\e)\right)\Big\|_{L_p(\mu)}^2\dd\mu(\e)\dd \mu(\eta)\bigg)^{\frac12}\\\gtrsim \bigg(\int_{\{-1,1\}^k}\int_{\{-1,1\}^k}S(\e,\d)^p\dd\mu(\e)\dd\mu(\d)\bigg)^{\frac{1}{p}}\stackrel{\eqref{eq:S lower}\wedge\eqref{eq:product measure estimate j}}{\gtrsim} \Big(\sum_{j=1}^k 2^{j(p-1)}\Big)^{\frac{1}{p}}\asymp 2^{\frac{k}{q}}\cdot \min\left\{q,k^{\frac{1}{p}}\right\},
\end{multline*}
which, when contrasted with~\eqref{eq:Mk upper}, implies~\eqref{eq:beta of ellp}.

A combination of~\eqref{eq:stp lower bound}, \eqref{eq:pisier m bound} and~\eqref{eq:beta of ellp} implies that there exist universal $a,b \in (0,1/2)$ such that
\begin{equation*}\label{eq:2 kapp}
\forall\, q\in [3,\infty),\qquad \beta(X)\ge a\min \left\{q,\big(q\log \left(1+bC_q(X)\right)\big)^{1-\frac{1}{q}}\right\}.
\end{equation*}
For every $q\ge \beta(X)/a$ this gives $a(q\log \left(1+bC_q(X)\right))^{2/3}\le a(q\log \left(1+bC_q(X)\right))^{1-1/q}\le \beta(X)$. Hence, $C_q(X)\le e^{3a/(2e)}/(ba^{3/2})$, since $q\ge \beta(X)/a$. By choosing $\kappa= \max\{1/a,e^{3a/(2e)}/(ba^{3/2})\}$, the proof of Lemma~\ref{lem:cotype bound} is complete.
\end{proof}

\begin{remark}\label{rem:np duality type}
{\em The same argument as in the proof of Lemma~\ref{lem:cotype bound}, without the need to use duality, shows that $T_{\kappa \beta/(\kappa\beta-1)}(Y)\le \kappa$. Since we shall not need this fact below, the details are omitted.}
\end{remark}

\subsection{UMD-valued Riesz potentials, Sobolev spaces and interpolation}\label{sec:bessel} Fix $n\in \N$ and $s,p\in (0,\infty)$. Suppose that $(Y,\|\cdot\|_Y)$ is a Banach space. If $f:\R^n\to Y$ is smooth and compactly supported then its homogeneous $(s,p)$-Riesz potential (semi)norm is defined as usual by
\begin{equation}\label{eq:def bessel}
\|f\|_{H_{s,p}(\R^n,Y)}\eqdef \left\|(-\Delta)^{\frac{s}{2}}f\right\|_{L_p(\R^n,Y)}=\left\|T_{\xi\mapsto \|\xi\|_2^s}f\right\|_{L_p(\R^n,Y)}.
\end{equation}
($\|\cdot\|_{H_{s,p}(\R^n,Y)}$ is sometimes denoted $\|\cdot\|_{\dot{H}_{s,p}(\R^n,Y)}$, but we use a simpler notation since  nonhomogeneous Riesz potentials do not occur in what follows.) The Banach space $H_{s,p}(\R^n,Y)$ is the completion of the  smooth and compactly support functions $f:\R^n\to Y$ under the norm $\|\cdot\|_{H_{s,p}(\R^n,Y)}$.

Throughout the ensuing discussion we shall use standard notation and basic facts from complex interpolation theory, as appearing in~\cite{BeLo}. The following lemma provides quantitative control on the behavior of the spaces $H_{s,p}(\R^n,Y)$ under complex interpolation when $Y$ is a UMD Banach space.

\begin{lemma}\label{lem:interpolation bessel} Let $(Y,\|\cdot\|_Y)$ be a UMD Banach space. Fix $p\in [1,\infty)$ and $s,\sigma\in (0,\infty)$ with $s<\sigma$. Suppose also that $\theta\in (0,1)$ and define $t=(1-\theta)s+\theta\sigma$. Then every $f\in H_{t,p}(\R^n,Y)$ satisfies
$$
\left\|f\right\|_{[H_{s,p}(\R^n,Y),H_{\sigma,p}(\R^n,Y)]_\theta}\lesssim \beta_p(Y)e^{\frac{\pi(\sigma-s)}{4}\sqrt{\theta(1-\theta)}}\left\|f\right\|_{H_{t,p}(\R^n,Y)}.
$$
\end{lemma}

\begin{proof} Consider the strip $S\eqdef \{z\in \C:\ \Re z\in (0,1)\}$. For every $M\in (0,\infty)$ define an auxiliary mapping $\Phi_M:\overline{S}\to H_{s,p}(\R^n,Y)+H_{\sigma,p}(\R^n,Y)$ by
$$
\forall\, z\in \overline S,\qquad \Phi_M(z)\eqdef e^{M(z(z-1)-\theta(\theta-1))}(-\Delta)^{\frac{t-s-z(\sigma-s)}{2}}f.
$$
Then $\Phi_M$ is holomorphic on $S$ and satisfies $\Phi_M(\theta)=f$. By the definition of the complex interpolation space $[H_{s,p}(\R^n,Y),H_{\sigma,p}(\R^n,Y)]_\theta$, we therefore have
\begin{equation}\label{eq:use def complex interpolation}
\left\|f\right\|_{[H_{s,p}(\R^n,Y),H_{\sigma,p}(\R^n,Y)]_\theta}\le \inf_{M>0} \sup_{b\in \R} \max\left\{ \|\Phi_M(ib)\|_{H_{s,p}(\R^n,Y)}, \|\Phi_M(1+ib)\|_{H_{\sigma,p}(\R^n,Y)}\right\}.
\end{equation}

For every $(a,b)\in [0,1]\times \R$ we have
$$
(-\Delta)^{\frac{s+a(\sigma-s)}{2}}\Phi_M(a-bi)=e^{Ma(a-1)+M\theta(1-\theta)-Mb^2-M(2a-1)bi}(-\Delta)^{\frac{ib(\sigma-s)}{2}} (-\Delta)^{\frac{t}{2}}f.
$$
Recalling~\eqref{eq:def bessel}, we therefore obtain the estimate
\begin{align}\label{eq:use imaginary laplacian}
\nonumber\|\Phi_M(a&+bi)\|_{H_{s+a(\sigma-s),p}(\R^n,Y)}\\\nonumber&\le e^{Ma(a-1)+M\theta(1-\theta)-Mb^2}\left\|(-\Delta)^{\frac{ib(\sigma-s)}{2}}\right\|_{L_p(\R^n,Y)\to L_p(\R^n,Y)}\|f\|_{H_{t,p}(\R^n,Y)}\\
&\lesssim \beta_p(Y)e^{Ma(a-1)+M\theta(1-\theta)}\cdot e^{-Mb^2+\frac{\pi |b|(\sigma-s)}{4}}\|f\|_{H_{t,p}(\R^n,Y)},
\end{align}
where in~\eqref{eq:use imaginary laplacian} we used Corollary~\ref{cor:imaginary power laplacian}. The function $b\mapsto -Mb^2+\pi |b|(\sigma-s)/4$ attains its maximum on $\R$ at $b=\pi (\sigma-s)/(8M)$. It therefore follows from~\eqref{eq:use imaginary laplacian} that
\begin{equation*}
\sup_{b\in \R}  \max\left\{\|\Phi_M(ib)\|_{H_{s,p}(\R^n,Y)},\|\Phi_M(1+ib)\|_{H_{\sigma,p}(\R^n,Y)}\right\}\lesssim \beta_p(Y)e^{M\theta(1-\theta)+\frac{\pi^2(\sigma-s)^2}{64M}}\|f\|_{H_{t,p}(\R^n,Y)}.
\end{equation*}
In combination with~\eqref{eq:use def complex interpolation} we therefore have
\begin{align}\label{eq:choose M}
\nonumber \left\|f\right\|_{[H_{s,p}(\R^n,Y),H_{\sigma,p}(\R^n,Y)]_\theta}&\lesssim \beta_p(Y)\left(\inf_{M\in (0,\infty)} e^{M\theta(1-\theta)+\frac{\pi^2(\sigma-s)^2}{64M}}\right)\|f\|_{H_{t,p}(\R^n,Y)}\\
&=\beta_p(Y)e^{\frac{\pi(\sigma-s)}{4}\sqrt{\theta(1-\theta)}}\left\|f\right\|_{H_{t,p}(\R^n,Y)},
\end{align}
where for~\eqref{eq:choose M} the optimal choice of $M\in (0,\infty)$ is $M=\pi(\sigma-s)/(8\sqrt{\theta(1-\theta)})$.
\end{proof}

Suppose that $\Omega\subset \R^n$ is open (for our purposes $\Omega$ will always be either a multiple of $B^n$ or all of $\R^n$). If $(Y,\|\cdot\|_Y)$ is a Banach space and  $p\in [1,\infty]$ then for every smooth   $f:\Omega \to Y$ denote
\begin{equation}\label{eq:Wp1 def}
\|f\|_{W_{1,p}(\Omega,Y)}\eqdef \sum_{j=1}^n \left\|\frac{\partial f}{\partial x_j}\right\|_{L_p(\Omega,Y)}.
\end{equation}
Thus $\|\cdot \|_{W_{p,1}(\Omega,Y)}$ is the (homogeneous) first order Sobolev (semi)norm of $f$.  The corresponding Sobolev space  $W_{p,1}(\Omega,Y)$ is the completion of the space of all smooth and compactly supported functions $f:\R^n\to Y$ under the norm $\|\cdot \|_{W_{p,1}(\Omega,Y)}$. For $(s,p)\in (0,1)\times [1,\infty)$, the order $s$ fractional (homogeneous) Sobolev (semi)norm of $f:\Omega \to Y$ is defined by
\begin{equation}\label{eq:s sobolev}
\|f\|_{W_{s,p}(\Omega,Y)}\eqdef \bigg(\iint_{\Omega\times \Omega}\frac{\|f(x)-f(y)\|_Y^p}{\|x-y\|_2^{n+ps}}\dd x\dd y\bigg)^{\frac{1}{p}}.
\end{equation}
While the notation $\|\cdot \|_{\dot{W}_{s,p}(\Omega,Y)}$ is sometimes used in the literature, we shall use the above simpler notation because nonhomogeneous Sobolev norms do not occur in what follows.

Our next goal is to relate UMD-valued Riesz potentials to  Sobolev norms. The following lemma treats the case of first order Sobolev norms, and also contains a comparison between Riesz potentials that will be needed later; it is a simple consequence of the boundedness of the Hilbert transform on UMD Banach spaces, combined with the method of rotations.

\begin{lemma}\label{lem:Hilbert}
Suppose that $(Y,\|\cdot\|_Y)$ is a UMD Banach space, $n\in \N$ and $p\in (1,\infty)$. Then every $f\in W_{1,p}(\R^n,Y)$ satisfies
$$
\|f\|_{H_{1,p}(\R^n,Y)}\le \beta_p(Y)^2 \|f\|_{W_{1,p}(\R^n,Y)}.
$$
Moreover, if $s\in (1,\infty)$ and $j\in \n$ then for every smooth $f:\R^n\to Y$ we have
\begin{equation}\label{eq:s-1,s}
\left\|\frac{\partial f}{\partial x_j}\right\|_{H_{s-1,p}(\R^n,Y)}\le \beta_p(Y)^2\|f\|_{H_{s,p}(\R^n,Y)}.
\end{equation}
\end{lemma}

\begin{proof}
Suppose that $K:\R^n\to Y$ is odd, continuous on $\R^n\setminus \{0\}$, and positively homogeneous of order $-n$, i.e., $K(tx)=t^{-n}K(x)$ for every $t\in (0,\infty)$ and $x\in \R^n\setminus \{0\}$. For $f\in L_p(\R^n,X)$ consider the corresponding  Calder\`on--Zygmund singular integral
\begin{equation}\label{eq:calderon zygmund}
\forall\, x\in \R^n,\qquad T_Kf(x)\eqdef \int_{\R^n} K(x-y) f(y)\dd y.
\end{equation}
It follows from the method of rotations, as presented by Iwaniec and Martin in~\cite{IM:Riesz}, that
\begin{equation}\label{eq:IM}
\|T_K\|_{L_p(\R^n,Y)\to L_p(\R^n,Y)} \le \frac{\pi}{2}\left(\int_{S^{n-1}} |K(z)|\dd\sigma(z)\right)\|H\|_{L_p(\R,Y)\to L_p(\R,Y)},
\end{equation}
where $\sigma$ is the surface area measure on the Euclidean sphere $S^{n-1}$ and $H:L_p(\R,Y)\to L_p(\R,Y)$ is the Hilbert transform, i.e.,
\begin{equation}\label{eq:def hilbert transform}
\forall\, \f\in L_p(\R,Y),\qquad H\f(x)\eqdef \frac{1}{\pi} \int_\R \frac{\f(y)}{x-y}\dd y.
\end{equation}
The integrals in~\eqref{eq:calderon zygmund} and~\eqref{eq:def hilbert transform} exist in the sense of principal values. The estimate~\eqref{eq:IM} is presented in~\cite[Proposition~5.1]{IM:Riesz} in the case $Y=\C$, but the same proof applies to the Banach space-valued setting without any change (the proof is based on an integral identity that is estimated using convexity of the norm. As such, the vector-valued and scalar-valued cases are identical; here the Banach space $Y$ can be general and the UMD property isn't used).

For $j\in \n$ consider the Riesz transform given by
$$
\forall\, x\in \R^n,\qquad R_jf(x)\eqdef \frac{\Gamma(\frac{n+1}{2})}{\pi^{\frac{n+1}{2}}} \int_{\R^n} \frac{x_j-y_j}{\|x-y\|_2^{n+1}}f(y)\dd y.
$$
By~\eqref{eq:IM} we have,
\begin{align}\label{eq:use IM}
\|R_j\|_{L_p(\R^n,Y)\to L_p(\R^n,Y)} &\le \frac{\pi\Gamma(\frac{n+1}{2})\int_{S^{n-1}} |z_1|\dd \sigma(z) }{2\pi^{\frac{n+1}{2}}} \|H\|_{L_p(\R,Y)\to L_p(\R,Y)}\nonumber\\ &=\|H\|_{L_p(\R,Y)\to L_p(\R,Y)} \le \beta_p(Y)^2,
\end{align}
where the bound $\|H\|_{L_p(\R,Y)\to L_p(\R,Y)}\le \beta_p(Y)^2$ that was used in~\eqref{eq:use IM} is implicit in the important work of Burkholder~\cite{Burkholder:83} and explicit in~\cite[Theorem~3]{Garling:86}.

Note that for every $j\in \n$ we have $R_j=(-\Delta)^{-1/2}\frac{\partial}{\partial x_j}$, as follows directly by computing the Fourier transform (see e.g.~\cite[Chapter~III]{Ste70-singular})). Consequently,
\begin{equation*}
\|f\|_{H_{1,p}(\R^n,Y)}\stackrel{\eqref{eq:def bessel}}{=}\left\|(-\Delta)^{-\frac12}\Delta f\right\|_{L_p(\R^n,Y)}=\Big\|\sum_{j=1}^n R_j\frac{\partial f}{\partial x_j}\Big\|_{L_p(\R^n,Y)}\stackrel{\eqref{eq:Wp1 def}\wedge \eqref{eq:use IM}}{\le} \beta_p(Y)^2\|f\|_{W_{1,p}(\R^n,Y)}.
\end{equation*}
Finally, to deduce~\eqref{eq:s-1,s}, proceed as follows.
\begin{multline*}
\left\|\frac{\partial f}{\partial x_j}\right\|_{H_{s-1,p}(\R^n,Y)}\stackrel{\eqref{eq:def bessel}}{=} \left\|(-\Delta)^{\frac{s-1}{2}}\frac{\partial f}{\partial x_j}\right\|_{L_p(\R^n,Y)}
= \left\|(-\Delta)^{-\frac12}\frac{\partial}{\partial x_j} (-\Delta)^{\frac{s}{2}}f\right\|_{L_p(\R^n,Y)}\\
=\left\|R_j(-\Delta)^{\frac{s}{2}}f\right\|_{L_p(\R^n,Y)}\stackrel{\eqref{eq:use IM}}{\le} \beta_p(Y)^2 \left\|(-\Delta)^{\frac{s}{2}}f\right\|_{L_p(\R^n,Y)}\stackrel{\eqref{eq:def bessel}}{=}
\beta_p(Y)^2\|f\|_{H_{s,p}(\R^n,Y)}. \tag*{\qedhere}
\end{multline*}
\end{proof}

The following theorem asserts a useful comparison between UMD-valued fractional Sobolev norms and the corresponding Riesz potentials.
\begin{theorem}\label{compare fractional}
Fix $n\in \N$, $s\in (0,1)$ and $p\in [2,\infty)$. Suppose that $(Y,\|\cdot\|_Y)$ is a UMD Banach space of cotype $p$. Then every $f\in H_{s,p}(\R^n,Y)$ satisfies
\begin{equation}\label{eq:fractional sobolev comparison}
\|f\|_{W_{s,p}(\R^n,Y)}\lesssim \frac{C_p(Y)\beta_p(Y)^{4}(nV_n)^{\frac{1}{p}} }{s(1-s)} \|f\|_{H_{s,p}(\R^n,Y)},
\end{equation}
\end{theorem}

\begin{proof}[Proof of Theorem~\ref{compare fractional}] 
The proof below proceeds via a reduction to a slightly stronger statement in the one dimensional case $n=1$. So, assume for the moment that we already proved that
\begin{equation}\label{eq:fractional sobolev comparison n=1}
\forall\, g\in H_{s,p}(\R,Y),\qquad \|g\|_{W_{s,p}(\R,Y)}\lesssim \frac{C_p(Y)\beta_p(Y)}{s(1-s)} \|g\|_{H_{s,p}(\R,Y)}.
\end{equation}
We shall now deduce the desired estimate~\eqref{eq:fractional sobolev comparison} from~\eqref{eq:fractional sobolev comparison n=1}, and then proceed to prove~\eqref{eq:fractional sobolev comparison n=1}.

For every $z\in S^{n-1}$ and $w\in z^\perp\subset \R^n$ define $g_{z,w}:\R\to Y$ by setting
\begin{equation}\label{eq:def gzw}
\forall\, t\in \R,\qquad g_{z,w}(t)\eqdef f(w+tz).
\end{equation}
By changing to polar coordinates we see that
\begin{eqnarray*}
\|f\|_{W_{s,p}(\R^n,Y)}^{p} &\stackrel{\eqref{eq:s sobolev}}{=}& \iint_{\R^n\times \R^n} \frac{\|f(x+y)-f(x)\|_Y^{p}}{\|y\|_2^{n+ps}}\dd x \dd y
\\&=&\frac12 \iiint_{S^{n-1}\times \R^n\times \R}\frac{\|f(x+rz)-f(x)\|_Y^{p}}{|r|^{1+ps}} \dd\sigma(z)\dd x \dd r\\&\stackrel{\eqref{eq:s sobolev}\wedge\eqref{eq:def gzw}}{=}&\frac12 \int_{S^{n-1}}\dd\sigma(z)\int_{z^{\perp}}\|g_{z,w}\|_{W_{s,p}(\R,Y)}^{p}\dd w.
\end{eqnarray*}
Hence, using~\eqref{eq:fractional sobolev comparison n=1} we deduce that
\begin{equation}\label{eq:after polar}
\|f\|_{W_{s,p}(\R^n,Y)}\lesssim \frac{C_p(Y)\beta_p(Y)}{s(1-s)} \left(\int_{S^{n-1}}\dd\sigma(z)\int_{z^{\perp}}\|g_{z,w}\|_{H_{s,p}(\R,Y)}^{p}\dd w\right)^{\frac{1}{p}}.
\end{equation}
For every $z\in S^{n-1}$, $w\in z^\perp$ and $t\in \R$ denote
$$
h_{z,w}(t)\eqdef \left(-\langle z,\nabla\rangle^2\right)^{\frac{s}{2}}f (w+tz)\stackrel{\eqref{eq:def gzw}}{=}\left(-\frac{\partial^2}{\partial t^2}\right)^{\frac{s}{2}}g_{z,w}(t).
$$
With this notation, for every $z\in S^{n-1}$ we have
\begin{multline}\label{eq:z dot}
\int_{z^{\perp}}\|g_{z,w}\|_{H_{s,p}(\R,Y)}^{p}\dd w\stackrel{\eqref{eq:def bessel}}{=}
\int_{z^{\perp}}
\left\|h_{z,w}\right\|_{L_{p}(\R,Y)}^{p}\dd w=\left\|\left(-\langle z,\nabla\rangle^2\right)^{\frac{s}{2}}f\right\|_{L_{p}(\R^n,Y)}^{p}\\
=\left\|\left(-\langle z,\nabla\rangle^2\right)^{\frac{s}{2}}(-\Delta)^{-\frac{s}{2}}(-\Delta)^{\frac{s}{2}}f\right\|_{L_{p}(\R^n,Y)}^{p}\stackrel{\eqref{eq:def bessel}}{\le} \left\|T_{m_z}\right\|_{L_{p}(\R^n,Y)\to L_{p}(\R^n,Y)}^{p}\left\|f\right\|_{H_{s,p}(\R^n,Y)}^{p},
\end{multline}
where we set $m_z(x)=|\langle z,x\rangle|^s/\|x\|_2^s$ for $x\in \R^n$ and $z\in S^{n-1}$. By rotation invariance, it follows from Corollary~\ref{cor:a power mult} that the norm of the multiplier $T_{m_z}$ as an operator from $L_{p}(\R^n,Y)$ to itself is at most a constant multiple of $\beta_p(Y)^3$. By combining this bound with~\eqref{eq:after polar} and~\eqref{eq:z dot}, we get
$$
\|f\|_{W_{s,p}(\R^n,Y)}\lesssim \frac{C_p(Y)\beta_p(Y)^4\sigma(S^{n-1})^{\frac{1}{p}}}{s(1-s)} \left\|f\right\|_{H_{s,p}(\R^n,Y)}.
$$
Since $\sigma(S^{n-1})=nV_n$, this concludes the deduction of~\eqref{eq:fractional sobolev comparison} from~\eqref{eq:fractional sobolev comparison n=1}.

It remains to prove~\eqref{eq:fractional sobolev comparison n=1}. The case $Y=\R$ of~\eqref{eq:fractional sobolev comparison n=1} is given in~\cite[Theorem~6.2.5]{BeLo}, without explicit dependence on the relevant parameters. The proof of~\eqref{eq:fractional sobolev comparison n=1} below consists of an adaptation of  the argument of~\cite[Theorem~6.2.5]{BeLo} to the UMD-valued setting, while tracking the bounds.

Recalling the Littlewood--Paley partition of unity $\{\psi_j\}_{j\in \Z}$ given in~\eqref{eq:def psik}, for every $y\in \R$ we have
\begin{align}
 \left(\int_\R\left\|g(x+y)-g(x)\right\|_Y^{p}\dd x\right)^{\frac{1}{p}}&\le \sum_{j\in \Z}\left(\int_\R\left\|T_{\psi_j}g(x+y)-T_{\psi_j(x)}g(x)\right\|_Y^{p}\dd x\right)^{\frac{1}{p}}\label{eq:use partition of unity}
\\&\lesssim \sum_{j\in \Z}  \min\Big\{1,\frac{|y|}{2^j}\Big\}\left\|T_{\psi_j}g\right\|_{L_{p}(\R,Y)}\label{eq:use localization},
\end{align}
where in~\eqref{eq:use partition of unity} we used the fact that $\sum_{j\in \Z}\psi_j\equiv 1$ and the triangle inequality in $L_{p}(\R,Y)$, and in~\eqref{eq:use localization} we used Lemma~\ref{lem:localization}. Now,
\begin{eqnarray}\label{eq:1 dim comparison used localization}
\|g\|_{W_{s,p}(\R,Y)}&\stackrel{\eqref{eq:use localization}}{\lesssim}&
\Big(\sum_{r\in \Z}\int_{2^r}^{2^{r+1}} \Big(\sum_{j\in \Z}  \min\Big\{1,\frac{|y|}{2^j}\Big\}\left\|T_{\psi_j}g\right\|_{L_{p}(\R,Y)}\Big)^{p}\frac{\dd y}{|y|^{1+ps}}\Big)^{\frac{1}{p}} \label{eq:integrand 2^r}\\
&\asymp& \Big(\sum_{r\in \Z}\Big(\sum_{u\in \Z}  \frac{\min\left\{1,2^{u}\right\}}{2^{su}}\cdot 2^{-s(r-u)}\left\|T_{\psi_{r-u}}g\right\|_{L_{p}(\R,Y)}\Big)^{p}\Big)^{\frac{1}{p}} \label{eq:j=r-u}\\
&\le & \sum_{u\in \Z} \Big(\sum_{r\in \Z}\frac{\min\left\{1,2^{pu}\right\}}{2^{psu}}\cdot 2^{-ps (r-u)}\left\|T_{\psi_{r-u}}g\right\|_{L_{p}(\R,Y)}^{p}\Big)^{\frac{1}{p}}\label{eq:triangle ell kappa beta}\\ \nonumber
&=& \Big(\sum_{j\in \Z} 2^{-ps j}\left\| T_{\psi_j}g\right\|_{L_{p}(\R,Y)}^{p}\Big)^{\frac{1}{p}}\sum_{u\in \Z}\frac{\min\left\{1,2^{ u}\right\}}{2^{su}}\\
&\asymp &\frac{1}{s(1-s)}\Big(\sum_{j\in \Z} 2^{-ps j}\left\| T_{\psi_j}g\right\|_{L_{p}(\R,Y)}^{p}\Big)^{\frac{1}{p}} \label{eq:geometric sum},
\end{eqnarray}
where for~\eqref{eq:j=r-u} use the fact that for each $r\in \Z$ in the integrand of the corresponding summand that appears in the right hand side of~\eqref{eq:integrand 2^r} we have $|y|\asymp 2^r$, and make the change of variable $u=r-j$, and for~\eqref{eq:triangle ell kappa beta} use the triangle inequality in $\ell_{p}(\Z)$.



Recalling  the functions $\{\t_j\}_{j\in \Z}$ given in~\eqref{eq:def thetha k}, an application of Proposition~\ref{prop:littlewood paley} shows that
\begin{equation}\label{eq:kappa beta cotype upper}
\E_{\e\in \{-1,1\}^{\Z}}\bigg[\Big\|\sum_{j\in \Z} \e_j T_{\t_j}(-\Delta)^{\frac{s}{2}}g\Big\|_{L_{p}(\R,Y)}\bigg]\le \beta_p(Y) \left\|(-\Delta)^{\frac{s}{2}}g\right\|_{L_{p}(\R,Y)}=\beta_p(Y)\|g\|_{H_{s,p}(\R,Y)}.
\end{equation}
At the same time, since $L_p(\R,Y)$ has cotype $p$ with constant $C_p(Y)$,
\begin{multline}\label{eq:kappa beta cotype lower}
\E_{\e\in \{-1,1\}^{\Z}}\bigg[\Big\|\sum_{j\in \Z} \e_j T_{\t_j}(-\Delta)^{\frac{s}{2}}g\Big\|_{L_{p}(\R,Y)}\bigg]\gtrsim \frac{1}{C_p(Y)}\Big(\sum_{j\in \Z} \left\| T_{\t_j}(-\Delta)^{\frac{s}{2}}g\right\|_{L_{p}(\R,Y)}^{p}\Big)^{\frac{1}{p}}\\
\gtrsim \frac{1}{C_p(Y)}\Big(\sum_{j\in \Z} \left\| T_{\psi_j}(-\Delta)^{\frac{s}{2}}g\right\|_{L_{p}(\R,Y)}^{p}\Big)^{\frac{1}{p}}
\gtrsim \frac{1}{C_p(Y)}\Big(\sum_{j\in \Z} 2^{-ps j}\left\| T_{\psi_j}g\right\|_{L_{p}(\R,Y)}^{p}\Big)^{\frac{1}{p}},
\end{multline}
where in the penultimate step of~\eqref{eq:kappa beta cotype lower} we used Lemma~\ref{psi theta comparison} and in the final step of~\eqref{eq:kappa beta cotype lower} we used Lemma~\ref{lem:localization2}. By combining~\eqref{eq:kappa beta cotype upper} and~\eqref{eq:kappa beta cotype lower}, we see that
\begin{equation}\label{eq:beta 18}
\Big(\sum_{j\in \Z} 2^{-ps j}\left\| T_{\psi_j}g\right\|_{L_{p}(\R,Y)}^{p}\Big)^{\frac{1}{p}}\lesssim C_p(Y)\beta_p(Y) \|g\|_{H_{s,p}(\R,Y)}.
\end{equation}
A substitution of~\eqref{eq:beta 18} into~\eqref{eq:geometric sum} now yields the desired estimate~\eqref{eq:fractional sobolev comparison n=1}.
\end{proof}

\begin{remark}{\em
The left side of \eqref{eq:beta 18} is the Besov norm $\|g\|_{B^s_{p,p}(\R,Y)}$. Hence~\eqref{eq:beta 18} asserts a quantitative embedding of $H_{s,p}(\R,Y)$ into $B^s_{p,p}(\R,Y)$ when $Y$ is a UMD space of cotype $p$. A qualitative embedding statement of this type was recently established by Veraar~\cite[Proposition 3.1]{Veraar:13}.
}\end{remark}

\begin{remark}\label{remark type} {\em One can prove a reverse inequality to that of Theorem~\ref{compare fractional} under the assumption that $(Y,\|\cdot\|_Y)$ is a UMD Banach space of type $p\in (1,2]$, in which case one obtains an upper bound on $\|f\|_{H_{s,p}(\R^n,Y)}$ in terms of $\|f\|_{W_{s,p}(\R^n,Y)}$. Specifically, we have the following estimate for  $s>0$.
\begin{equation}\label{eq:type case}
\forall f\in W_{s,p}(\R^n,Y),\qquad \|f\|_{H_{s,p}(\R^n,Y)}\lesssim \frac{T_p(Y)\beta_p(Y)^{4}n^{\frac{s}{2}}}{(p-1)(nV_n)^{1-\frac{1}{p}}} \|f\|_{W_{s,p}(\R^n,Y)}.
\end{equation}
Since~\eqref{eq:type case} is not needed below, we omit its proof (which is available on request). By Remark~\ref{rem:np duality type}, one can take in~\eqref{eq:type case} $p-1\asymp 1/\beta(Y)$, in which case $T_p(Y)\asymp 1$ and, by~\eqref{eq:betap beta 2}, $\beta_p(Y)\lesssim \beta(Y)^2$. Note that, by Kwapien's theorem~\cite{Kwa72}, unless $Y$ is isomorphic to a Hilbert space, its type and cotype do not coincide, and therefore by combining Theorem~\ref{compare fractional} with~\eqref{eq:type case} one does not obtain an equivalence between the norms $\|\cdot \|_{H_{s,p}(\R^n,Y)}$ and $\|\cdot \|_{W_{s,p}(\R^n,Y)}$ for non-Hilbertian targets $Y$. See~\cite{HM96} for a related characterization of Hilbert space.
}
\end{remark}

Recalling the notation that was introduced in Section~\ref{sec:prelim}, we end this section by deducing a corollary of Theorem~\ref{compare fractional} that will be very important in what follows.

\begin{corollary}\label{cor:combine bessel potential with legendre} Fix $p\in [2,\infty)$ and $(s,\sigma)\in (0,1)\times (1,2)$. Let $(Y,\|\cdot\|_Y)$ be a UMD Banach space of cotype $p$. Then every measurable $f:\R^n\to Y$ satisfies
\begin{equation}\label{eq:coro s}
 \bigg(\iiint_{ \R^n\times B^n\times (0,\infty)}\frac{\|f^x(uy)-P^1_uf^x(uy)\|_Y^p}{V_nu^{ps+1}}\dd x\dd y\dd u\bigg)^{\frac{1}{p}} \lesssim \frac{\sqrt{pn} C_p(Y)\beta_p(Y)^{4} }{s(1-s)}\|f\|_{H_{s,p}(\R^n,Y)},
\end{equation}
and
\begin{equation}\label{eq:coro sigma}
 \bigg(\iiint_{ \R^n\times B^n\times (0,\infty)}\frac{\|f^x(uy)-P^1_uf^x(uy)\|_Y^p}{V_nu^{p\sigma+1}}\dd x\dd y\dd u\bigg)^{\frac{1}{p}} \lesssim \frac{\sqrt{p}n^{\frac32}C_p(Y)\beta_p(Y)^{6}}{(\sigma-1)(2-\sigma)}\|f\|_{H_{\sigma,p}(\R^n,Y)}.
\end{equation}
\end{corollary}

\begin{proof} By integrating the conclusion of Corollary~\ref{lem:fractional sobolev appears} (with $q=ps$) over $x\in \R^n$ we see that
\begin{equation}\label{eq:integrate corollary1}
 \bigg(\frac{1}{V_n}\iiint_{ \R^n\times B^n\times (0,\infty)}\frac{\|f^x(uy)-P^1_uf^x(uy)\|_Y^p}{u^{ps+1}}\dd x\dd y\dd u\bigg)^{\frac{1}{p}}\lesssim  \frac{\sqrt{pn}}{(nV_n)^{\frac{1}{p}}}\|f\|_{W_{s,p}(\R^n,Y)}.
\end{equation}
A substitution of the conclusion of Theorem~\ref{compare fractional} into~\eqref{eq:integrate corollary1} yields~\eqref{eq:coro s}.

Next, by integrating the conclusion of Corollary~\ref{lem:fractional sobolev appears2} (with $q=p\sigma>p$) over $x\in \R^n$ we see that
\begin{equation}\label{eq:integrate corollary2}
 \bigg(\frac{1}{V_n}\iiint_{ \R^n\times B^n\times (0,\infty)}\frac{\|f^x(uy)-P^1_uf^x(uy)\|_Y^p}{u^{p\sigma+1}}\dd x\dd y\dd u\bigg)^{\frac{1}{p}}\lesssim \frac{\sqrt{pn}}{(nV_n)^{\frac{1}{p}}}\sum_{j=1}^n\left\|\frac{\partial f}{\partial x_j}\right\|_{W_{\sigma-1,p}(\R^n,Y)}.
\end{equation}
By Theorem~\ref{compare fractional} applied to  $\frac{\partial f}{\partial x_j}$ for each $j\in \n$ and with $s$ replaced by $\sigma-1\in (0,1)$,
\begin{equation}\label{eq:use main compare sigma-1}
\forall\, j\in \n,\qquad \left\|\frac{\partial f}{\partial x_j}\right\|_{W_{\sigma-1,p}(\R^n,Y)}\lesssim \frac{C_p(Y)\beta_p(Y)^{4}(nV_n)^{\frac{1}{p}} }{(\sigma-1)(2-\sigma)}\left\|\frac{\partial f}{\partial x_j}\right\|_{H_{\sigma-1,p}(\R^n,Y)}.
\end{equation}
Moreover, by the second assertion of Lemma~\ref{lem:Hilbert} we have
\begin{equation}\label{eq:use derivative compare sigma}
\max_{j\in \n}\left\|\frac{\partial f}{\partial x_j}\right\|_{H_{\sigma-1,p}(\R^n,Y)} \lesssim \beta_p^2(Y) \left\|f\right\|_{H_{\sigma,p}(\R^n,Y)}.
\end{equation}
By substituting~\eqref{eq:use derivative compare sigma} into~\eqref{eq:use main compare sigma-1}, and then substituting the resulting estimate into~\eqref{eq:integrate corollary2}, we obtain the desired estimate~\eqref{eq:coro sigma}.
\end{proof}

\section{Proof of Theorem~\ref{w1p dorronsoro}}\label{sec:w1 proof}

We shall prove here the following theorem.

\begin{theorem}\label{thM:doro more refined}
Fix $p\in [2,\infty)$ and suppose that $(Y,\|\cdot\|_Y)$ is a UMD Banach space of cotype $p$. Then every smooth $f: \R^n\to Y$ satisfies
 \begin{multline}\label{eq:doro refined}
\bigg(\iiint_{\R^n\times B^n\times (0,\infty)}\frac{\left\|f^x(uy)-P_u^1f^x(uy)\right\|_Y^{p}}{V_n  u^{p+1}}\dd x  \dd y\dd u\bigg)^{\frac{1}{p}}
\\\lesssim \sqrt{pn} C_p(Y)\beta_p(Y)^{7}\log(\beta_p(Y)n) \sum_{j=1}^n \left(\int_{\R^n}\Big\|\frac{\partial f}{\partial x_j}(x)\Big\|_Y^{p}\dd x\right)^{\frac{1}{p}}.
\end{multline}
\end{theorem}

Note that Theorem~\ref{w1p dorronsoro} follows from Theorem~\ref{thM:doro more refined}. Indeed, if $(Y,\|\cdot\|_Y)$ is a UMD Banach space with $\beta=\beta(Y)$ then by Lemma~\ref{lem:cotype bound} there exists a universal constant $\kappa\in (0,\infty)$ such that if we set $p=\kappa\beta$ then $C_p(Y)\lesssim 1$ (in particular we necessarily have $\kappa\beta\ge 2$).
Moreover, by~\eqref{eq:betap beta 2} we have $\beta_p(Y)\lesssim \beta^2$. To deduce Theorem~\ref{w1p dorronsoro}, take $f:\R^n\to Y$ that is Lipschitz and compactly supported. By convolving $f$ with a smooth bump function whose support has small diameter we may also assume that $f$ is smooth. It now follows from Theorem~\ref{thM:doro more refined} that
 \begin{multline*}
\bigg(\iiint_{\R^n\times B^n\times (0,\infty)}\frac{\left\|f^x(uy)-P_u^1f^x(uy)\right\|_Y^{\kappa\beta}}{V_n  u^{\kappa\beta+1}}\dd x  \dd y\dd u\bigg)^{\frac{1}{\kappa\beta}}
\\ \lesssim \beta^{\frac{29}{2}}\sqrt{n}\log (\beta n) \sum_{j=1}^n \left(\int_{\R^n}\Big\|\frac{\partial f}{\partial x_j}(x)\Big\|_Y^{\kappa\beta}\dd x\right)^{\frac{1}{\kappa\beta}}\le \beta^{15}n^{\frac52}\big(\vol(\supp(f))\big)^{\frac{1}{\kappa\beta}}\|f\|_{\Lip},
\end{multline*}
where we used the fact that, due to~\eqref{eq:john posiiton}, for every $x\in \R^n$ and $j\in \n$ we have
$$
\Big\|\frac{\partial f}{\partial x_j}(x)\Big\|_Y=\lim_{t\to 0}\frac{\|f(x+te_j)-f(x)\|_Y}{|t|}\le \|f\|_{\Lip}\|e_j\|_X\stackrel{\eqref{eq:john posiiton}}{\le} \|f\|_{\Lip}\sqrt{n}.
$$

\begin{proof}[Proof of Theorem~\ref{thM:doro more refined}]For every $s\in (0,\infty)$ let $\nu_s$ be the measure on $(0,\infty)\times \R^n\times B^n$ whose density is given by
\begin{equation}\label{eq:def density nus}
\forall(u,x,y)\in (0,\infty)\times \R^n\times B^n,\qquad \f_s(u,x,y)\eqdef \frac{1}{V_n|u|^{ps+1}}.
\end{equation}
Define a linear operator $\mathscr{S}:H_{s,p}(\R^n,Y)\to L_{p}(\nu_s,Y)$ by setting for every $f\in H_{s,p}(\R^n,Y)$,
\begin{equation}\label{eq:def operator S}
\forall(u,x,y)\in (0,\infty)\times \R^n\times B^n,\qquad \mathscr{S}f(u,x,y)\eqdef f^x(uy)-P^1_uf^x(uy),
\end{equation}
where we recall the notations~\eqref{eq:def translate} and \eqref{eq:def P1} for $f^x$ and $P^1_u$, respectively. In what follows we let $M_s$ denote the norm of $\mathscr{S}$ as an operator from $H_{s,p}(\R^n,Y)$ to $L_{p}(\nu_s,Y)$, i.e.,
\begin{equation*}\label{eq:def Ms}
\forall\, s\in (0,\infty),\qquad M_s\eqdef  \|\mathscr S\|_{H_{s,p}(\R^n,Y)\to L_{p}(\nu_s,Y)}.
\end{equation*}

Suppose that $s,\sigma,\theta\in \R$ satisfy
\begin{equation*}\label{eq:s sigma theta assumptions}
(s,\sigma,\theta)\in (0,1)\times (1,2)\times (0,1)\qquad\mathrm{and}\qquad 1=(1-\theta) s+\theta\sigma.
\end{equation*}
Then $\f_1=\f_s^{1-\theta} \f_\sigma^{\theta}$, so by Stein's interpolation theorem~\cite[Theorem~2]{Ste56}  for every $f$ in the complex interpolation space $[H_{s,p}(\R^n,Y),H_{\sigma,p}(\R^n,Y)]_\theta$ we have
\begin{multline}\label{eq:interpolation ineq S}
\bigg(\iiint_{\R^n\times B^n\times (0,\infty)}\frac{\left\|f^x(uy)-P_u^1f^x(uy)\right\|_Y^{p}}{V_n  u^{p+1}}\dd x  \dd y\dd u\bigg)^{\frac{1}{p}}\\
\stackrel{\eqref{eq:def density nus}\wedge\eqref{eq:def operator S}}{=} \|\mathscr{S}f\|_{L_{p}(\nu_1,Y)} \le M_s^{1-\theta}M_\sigma^{\theta}\|f\|_{[H_{s,p}(\R^n,Y),H_{\sigma,p}(\R^n,Y)]_\theta}.
\end{multline}
We note that Stein's interpolation theorem is stated in~\cite{Ste56} for real-valued function spaces, but the standard proofs of this theorem (see also~\cite{SW58} or~\cite[Section~5.4]{BeLo})  work without additional effort for vector-valued spaces as well, which is what we are using here. Alternatively, the vector-valued setting is treated explicitly by  Calder\'on in~\cite[Section~13.6]{Cal64}. Every $f\in W_{1,p}(\R^n,Y)$ satisfies
\begin{equation}\label{eq:use comparisons of interpolation}
\|f\|_{[H_{s,p}(\R^n,Y),H_{\sigma,p}(\R^n,Y)]_\theta}\stackrel{(*)}{\lesssim} \beta_p(Y)\|f\|_{H_{1,p}(\R^n,Y)}\stackrel{(**)}{\lesssim} \beta_p(Y)^3\|f\|_{W_{1,p}(\R^n,Y)}.
\end{equation}
where in $(*)$ we used Lemma~\ref{lem:interpolation bessel} and in $(**)$ we used Lemma~\ref{lem:Hilbert}. Hence,
\begin{multline}\label{almost doro thm}
\bigg(\iiint_{\R^n\times B^n\times (0,\infty)}\frac{\left\|f^x(uy)-P_u^1f^x(uy)\right\|_Y^{p}}{V_n  u^{p+1}}\dd x  \dd y\dd u\bigg)^{\frac{1}{p}}
\\\stackrel{\eqref{eq:Wp1 def}\wedge \eqref{eq:interpolation ineq S}\wedge\eqref{eq:use comparisons of interpolation}}{\lesssim} \beta_p(Y)^3\bigg(\inf_{\stackrel{(s,\sigma,\theta)\in (0,1)\times (1,2)\times (0,1)}{(1-\theta)s+\theta \sigma=1}}M_s^{1-\theta}M_\sigma^{\theta}\bigg)\sum_{j=1}^n \left(\int_{\R^n}\Big\|\frac{\partial f}{\partial x_j}(x)\Big\|_Y^{p}\dd x\right)^{\frac{1}{p}}.
\end{multline}

Corollary~\ref{cor:combine bessel potential with legendre} asserts that
$$
\forall(s,\sigma)\in (0,1)\times (1,2), \qquad M_s\lesssim \frac{\sqrt{pn} C_p(Y)\beta_p(Y)^{4} }{s(1-s)}\qquad\mathrm{and}\qquad M_\sigma\lesssim   \frac{\sqrt{p}n^{\frac32}C_p(Y)\beta_p(Y)^{6}}{(\sigma-1)(2-\sigma)}.
$$
Hence,
\begin{align}\label{eq:optimize interpolation}
\inf_{\stackrel{(s,\sigma,\theta)\in (0,1)\times (1,2)\times (0,1)}{(1-\theta)s+\theta \sigma=1}}M_s^{1-\theta}M_\sigma^{\theta} \nonumber&\lesssim
\inf_{\stackrel{(s,\sigma,\theta)\in (0,1)\times (1,2)\times (0,1)}{(1-\theta)s+\theta \sigma=1}}\bigg(\frac{\sqrt{pn} C_p(Y)\beta_p(Y)^{4} }{s(1-s)}\bigg)^{1-\theta}\bigg(\frac{\sqrt{p}n^{\frac32}C_p(Y)\beta_p(Y)^{6}}{(\sigma-1)(2-\sigma)}\bigg)^\theta\\
&\lesssim \sqrt{pn} C_p(Y)\beta_p(Y)^{4}\log(\beta_p(Y)n),
\end{align}
where~\eqref{eq:optimize interpolation} arises by choosing $s=1-1/\log(\beta_p(Y)n)$ and $\sigma=2-1/\log(\beta_p(Y)n)$, in which case necessarily $\theta= 1/\log(\beta_p(Y)n)$. We note that these choices essentially yield the best possible estimate in~\eqref{eq:optimize interpolation} (up to constant factors), but one could also choose here, say, $s=1/2$ and $\sigma=3/2$, yielding a worse dependence on $n$ which is of lesser importance for our present purposes. A substitution of~\eqref{eq:optimize interpolation} into~\eqref{almost doro thm} yields the desired estimate~\eqref{eq:doro refined}.
\end{proof}

\subsection*{Acknowledgements} We are grateful to Charles Fefferman for allowing us to include the contents of Section~\ref{sec:upper bounds}. We also thank Mark Veraar for helpful pointers to the literature and we are grateful to the anonymous referee for a careful reading and detailed comments that improved the presentation.

Part of this work was completed during a visit of the three authors to the MSRI Quantitative Geometry program. Another part of this work was completed when A.~N. was visiting Universit\'e Pierre et Marie Curie, Paris, France.

\subsection*{Added in proof} In the forthcoming work~\cite{HN16}, Question~\ref{Q:UC} is resolved positively by showing that Theorem~\ref{thm:UAAP intro} holds true for  any uniformly convex target $Y$ (in which case the parameter $\beta$ is replaced by a quantity that depends on the modulus of uniform convexity of $Y$). Also, the dependence on $n$ that appears in~\eqref{eq:n4} is improved in~\cite{HN16}. This is achieved in~\cite{HN16} by following the vector-valued  Littlewood--Paley strategy that we introduced here to bound $r^{X\to Y}(\e)$, but while implementing it via a method that differs markedly from the argument that appears in the present work. Specifically, \cite{HN16} follows more closely Bourgain's original strategy~\cite{Bou87} for proving his discretization theorem, though with major  differences. In particular, the proof in~\cite{HN16}  even yields a new approach to Dorronsoro's influential classical work~\cite{Do} in the scalar-valued setting.

\bibliographystyle{abbrv}
\bibliography{Quant-UAAP-UMD}


\begin{dajauthors}

\begin{authorinfo}[tuomas]

Tuomas Hyt\"{o}nen\\
University of Helsinki\\
 Helsinki, Finland\\
tuomas\imagedot{}hytonen\imageat{}helsinki\imagedot{}fi\\
  \url{http://wiki.helsinki.fi/pages/viewpage.action?pageId=31427518}
\end{authorinfo}
\begin{authorinfo}[sean]
Sean Li\\
University of Chicago, Chicago IL, USA\\
seanli\imagedot{}math\imageat{}uchicago\imagedot{}edu\\
  \url{https://sites.google.com/site/seanlimath/}
\end{authorinfo}
\begin{authorinfo}[assaf]
  Assaf Naor\\
 Princeton University\\
  Princeton NJ, USA\\
  naor\imageat{}math\imagedot{}princeton\imagedot{}edu\\
  \url{https://web.math.princeton.edu/~naor/}
\end{authorinfo}
\end{dajauthors}

\end{document}
